\documentclass{amsart}
\newtheorem{theorem}{Theorem}[section]
\newtheorem{lemma}[theorem]{Lemma}
\theoremstyle{definition}
\newtheorem{definition}[theorem]{Definition}

\newtheorem{corollary}[theorem]{Corollary}

\theoremstyle{remark}
\newtheorem{remark}[theorem]{Remark}
\theoremstyle{proposition}
\newtheorem{proposition}[theorem]{Proposition}
\numberwithin{equation}{section}
\usepackage[all]{xy}
\usepackage{amsmath,mathrsfs}
\usepackage{graphics}
\usepackage{pstricks}
\title{Hartogs' phenomenon and the problem of complex spheres}
\author{ Zouaoui Mekri }
\address{Zouaoui MEKRI.  High school of computer science\\
22000 Sidi bel-abbes (Algeria)}
\email{z.mekri@esi-sba.dz}
\email{zouaouimekri@gmail.com}
\begin{document}

\maketitle
\begin{center}
    In memory of Pierre Dolbeault (1924-2015)
\end{center}

$ $\\\textbf{Abstract}\\In this paper, we solve in the negative the following open problem.\\
\textbf{\underline{Problem of complex spheres}.}$\phantom{r}$
\emph{Is there any integrable almost complex structure on the sphere $\mathcal{S}_{6}$?}
\tableofcontents
\begin{center}
    \textbf{Introduction}
\end{center}
$ $\\
The story of this long standing problem began more than sixty years ago when C. Ehresmann has introduced in differential geometry the
notion of almost complex structure on a differentiable manifold of even dimension (see \cite{5}, \cite{6}). The first result in the
direction of this problem was given by Borel and Serre \cite{1} and we know according to this result  that the only spheres of even dimension
which admit an almost complex structure are $\mathcal{S}_{2}$ and $\mathcal{S}_{6}$; therefore no spheres of even dimension except perhaps
$\mathcal{S}_{2}$ and $\mathcal{S}_{6}$ admit a complex structure. Since it is well known that $\mathcal{S}_{2}$ is a Riemann surface, that is
already a complex manifold of $dim_{\mathbb{C}}\mathcal{S}_{2} = 1$, it remained then to study only the case $\mathcal{S}_{6}$. It was Ehresmann
himself (see \cite{5}, \cite{16}) who showed for the first time that the sphere $\mathcal{S}_{6}$ admitted almost complex structures, and it was
Kirchoff (see \cite{13}, \cite{16}) who gave  the first example of an almost complex structure on $\mathcal{S}_{6}$. This example is the standard
almost complex structure $J : T\mathcal{S}_{6} \longrightarrow T\mathcal{S}_{6}$ defined by the cross product $\times$ on imaginary
parts $u , v \in \mathbb{R}^{7}$ of the Cayley's Octonions
$$
T_{u}(v) = u \times v  \phantom{rrrrr} \textrm{ for all } \phantom{r} v \in T_{u}S_{6} = u^{\perp} \subset \mathbb{R}^{7}
.$$ But it turned out that this famous almost complex structure is not integrable. The question was then: How to determine whether there exists or not
an integrable almost complex structure on the sphere $\mathcal{S}_{6}$? One great difficulty in solving this question is that the theorem of
 integrability of almost complex structures of Newlander Nirenberg \cite{17}, can not be easily implemented.
We mention however two kinds of fundamental partial results in the direction of the problem of the "complex" sphere $\mathcal{S}_{6}$.
The first result shows that, roughly speaking, although the sphere $\mathcal{S}_{6}$ is "simple" as real analytic manifold, an hypothetic
 complex structure makes it more "complicated" as complex analytic manifold. Indeed, we know that the Betti numbers $b_{j}$ of
 $\mathcal{S}_{6}$ with $ 1 \leq j \leq 5$ are null, however a surprising result due to A. gray \cite{8},(see also Ughart \cite{19}), shows
 that the Hodge number $h^{0,1}$ of an hypothetic complex structure on $\mathcal{S}_{6}$, is $\neq 0$, which means that not any $\overline{\partial}-$closed
  differential form of bidegree  $(0,1)$ in $\mathcal{S}_{6}$ is $\overline{\partial}-$exact; or in other words, the Cauchy-Riemann
  equation $\overline{\partial}u = f$ can not be always resoluble in the whole $\mathcal{S}_{6}$!
The second kind of results shows that some almost complex structures on $\mathcal{S}_{6}$ which satisfy some "additional conditions" can
also not be integrable, see \cite{2}, \cite{3}, \cite{15}. For instance, C. LeBrun has shown in \cite{15}, that there exist no integrable almost
complex structures $J$ on $\mathcal{S}_{6}$ orthogonal with respect to the Riemannian metric of $\mathcal{S}_{6}$.
 Such was the strategy adopted by almost all geometers interested in this problem!, and all attempts to solve the problem of the existence (or not)
 of a complex structure on the sphere $\mathcal{S}_{6}$, have been partially unsuccessful. The question remains then still open!
In this paper, we will not be occupied  by the integrability of any a priori almost complex structure on $\mathcal{S}_{6}$, since we will not have to
 use the result of Borel and Serre \cite{1}. Indeed, we propose an approach which shows that for $n \geq 2$, the obstruction for the
  spheres $\mathcal{S}_{2n}$  to admit a complex structure, is not due only to the fact that they do not admit almost complex structures,
   as is the case for the spheres other than $\mathcal{S}_{6}$, but that this phenomenon is due to Hartogs' theorem which is valid for all
    $n \geq 2$; and then, even if the the spheres $\mathcal{S}_{2n}$ with $n \geq 2$ other than $\mathcal{S}_{6}$ admitted almost complex
    structures, they do in any way, admit a complex structure. Although we adopt a cohomological approach, we don't search to apply the
     $\overline{\partial}-$cohomology of the whole sphere $\mathcal{S}_{6}$ as in \cite{8} and \cite{19}, because this turned out to be
     not conclusive. We will rather focus our attention on the study of the Dolbeault cohomology of a family of hypothetic open complex
      submanifolds $X_{\alpha,\beta}$ of $\mathcal{S}_{2n}$, called $\mathcal{A}-$submanifolds of $\mathcal{S}_{2n}$
      (see definition \ref{A-sous variete de S_2n}), and then we look for the problem of complex spheres from the point of view of Hartogs
      phenomenon.  Indeed, and surprisingly enough, it turns out that the problem of complex spheres on $\mathcal{S}_{2n}$, viewed
      under $\check{C}$hech cohomology, is a geometric phenomenon, deeply linked to the problem of analytic continuation
      of holomorphic functions  defined in a connected domain around a compact subset of $\mathbb{C}^{n}$ with $n \geq 2$, that
      is, to the so-called Hartogs phenomenon! \cite{10}, \cite{11}, \cite{18}. We prove in this paper two fundamental results
      concerning the so called $\mathcal{A}-$submanifolds of $\mathcal{S}_{2n}$ mentioned above: The first result shows that from
      the topological point of view, any $\mathcal{A}-$submanifolds of $\mathcal{S}_{2n}$ is homeomorphic to $\mathbb{C}^{n}$, and
      the second one shows that any $\mathcal{A}-$submanifolds of $\mathcal{S}_{2n}$ is strictly pseudoconvex, and then by the well
      known Cartan's theorem B, the Dolbeault  group of cohomology of bidegree $(0,1)$ of such manifold vanishes, that is
      $$\mathcal{H}^{0,1}\left( X_{\alpha,\beta}, \mathbb{C}^{n}\right) = \left\{ 0 \right\}.$$ Then equipped with this vanishing result,
       we show using $\check{C}$ech cohomology, that the assumption that $\mathcal{S}_{2n}$ with $n \geq 2$, admits a complex structure is
       in contradiction with Hartogs' theorem, which gives a negative solution to the problem of complex spheres for all spheres
       $\mathcal{S}_{2n}$ with $ n \geq 2$ and then for $\mathcal{S}_{6}$ as a particular case (\footnote{The approach proposed in this paper
       does not suppose the result of Borel Serre, and then we consider the problem of complex spheres for all spheres $\mathcal{S}_{2n}$ with
       $n \geq 2$ and not only for $\mathcal{S}_{6}$.}). The problem of complex spheres becomes then a geometric phenomenon induced by the
       analytic Hartogs' phenomenon, and we prove even that both phenomenons are in fact equivalent. It is rather curious that for more than a
       century, we have always viewed  Hartog's phenomenon only from its analytic aspect! We will see in this paper, that Hartogs' phenomenon and
       the problem of complex spheres represent two faces of the same phenomenon, which answers in the negative the not yet solved problem of
       complex spheres. The proof of such a link is the main motivation for this work.
\section{\textbf{Statement of results.}}$ $\\
\subsubsection{\underline{\textbf{Notations.}}}\label{Notations} It is useful before beginning our exposition to specify once for all some frequent notations.
\begin{enumerate}
  \item  $\mathcal{B}(a,r)$ denotes always an open ball of $\mathbb{C}^{n}$ centered at $a$, and of radius $r > 0$, and when we note
  simply $\mathcal{B}_{r}$ without specifying the center, this always means that the ball
is centered at the origin. The corresponding closed balls, will be noted respectively by $\overline{\mathcal{B}}(a,r)$,
 and $\overline{\mathcal{B}}_{r}$.
\item  $\mathcal{S}_{k}(a,r)$ denotes a sphere of dimension $k < 2n$ embedded in $\mathbb{C}^{n}$, of radius $r > 0$ and
centered at the point $a \in \mathbb{C}^{n}$, and when we note $\mathcal{S}_{k}$, this always means that the sphere
 embedded in $\mathbb{C}^{n}$, is viewed as a unit sphere centered at the origin.
  \item For two subsets $X $ and $Y$ of $\mathbb{C}^{n}$, we note
$$
Y - X  = \bigg\{ z \in Y, \phantom{rr}z\not\in X\cap Y\bigg\}.$$
\item If $X$ and $Y$ are topological spaces, the notation $X \approx Y$ means always that $X$ is homeomorphic to $Y$.
\item  We use the abbreviate notation "\emph{\textbf{psh}}" for plurisubharmonic functions.
\item If $X$ is a complex manifold and $f: X \longrightarrow\mathbb{C}$ is any piecewise $\mathcal{C}^{k}-$function, we note
then by $\widetilde{f}$
the expression of $f$ in a chart $(V,\psi)$, that is, in the complex coordinates $\zeta = \psi(z) \in \mathbb{C}^{n}$
 \begin{equation*}
                    \xymatrix { V \ar[d]_{\psi} \ar[r]^{f}& \mathbb{K}  \\
\psi\left( V\right) \ar[ur]_{\widetilde{f}} & }
                 \end{equation*}
\end{enumerate}
\subsection{Statement of the theorem of complex spheres}$ $\\
We begin by stating the following theorem.
\begin{theorem}\label{spheres}\textbf{(Theorem of complex spheres)}$ $\\
Among all spheres of even dimension $S_{2n}$, only the sphere $S_{2}$ admits a complex structure.
\end{theorem}
 One of the fundamental problems in several complex variables solved in the last century (period 1942-1954) was the so called Levi's problem which
 characterizers domains of holomorphy in $\mathbb{C}^{n}$. For the commodity of the reader, we recall here that a domain of holomorphy is roughly speaking
  a maximum domain $D \subseteq \mathbb{C}^{n}$ of definition of a holomorphic function $f: D \longrightarrow \mathbb{C}$. (By "maximum" we mean that
   $f$ can not be analytically extended across any part of the boundary $\partial D$). Two fundamental results in several complex analysis marked the
    beginning of the 20th century,  both proved in 1906. One of them (\footnote{The second result is the theorem proved by H. Poincar\'{e} which says
    that in $\mathbb{C}^{n}$ with $n \geq 2$, a ball is never biholomorphic to a polydisc.}) shows that not any domain in several complex variables
     is a domain of holomorphy. This result is the now classical Hartogs' theorem, See \cite{4}, \cite{7}, \cite{10}, \cite{12}, \cite{15}. We will
      prove that the theorem of complex spheres stated above is a corollary of the main theorem of this paper (theorem \ref{equiv}) combined with Hartogs'
      theorem, whose statement is as follows:
\begin{theorem}\label{hartogs}\textbf{(Hartogs' theorem)} $ $ \\
Let $n \geq 2$, and let $D  \subset \mathbb{C}^{n}$ be an open set, and let $\mathcal{K}$ be a compact subset of $D$, such that $D - \mathcal{K}$ is
connected. Then for every holomorphic mapping $f: D - \mathcal{K} \longrightarrow \mathbb{C}^{k},$ there exists a holomorphic mapping
$F: D \longrightarrow \mathbb{C}^{k}$ such that $F=f$ in $D - \mathcal{K}$.
\end{theorem}
Before stating our main theorem (theorem \ref{equiv}), let us first consider the following two propositions:
\begin{proposition}\label{sphere sphere} \textbf{$\left\{Hath\right\}_{n}$} $ $ \\
Let $n \in \mathbb{N}^{\ast}$, and let $\mathcal{B}\subset \mathbb{C}^{n}$ be an open ball and let $\mathcal{K}$ be a compact subset of $\mathcal{B}$,
 such that $\mathcal{B}-\mathcal{K}$ is connected. Then for every holomorphic mapping $f: \mathcal{B} - \mathcal{K} \longrightarrow \mathbb{C}^{k},$
  there exists a holomorphic mapping $F: \mathcal{B} \longrightarrow \mathbb{C}^{k}$ such that $F=f$ in $\mathcal{B} - \mathcal{K}$.
\end{proposition}
and
\begin{proposition}\label{prop2}\textbf{$\left\{Necs\right\}_{n}$}$ $\\
Let $n \in \mathbb{N}^{\ast}$, then the sphere $S_{2n}$ does not admit a complex structure.
\end{proposition}
\begin{remark}$ $\\
\begin{enumerate}
  \item In this paper, we will refer to propositions \ref{sphere sphere}, and \ref{prop2} respectively, by the following abbreviated notations:
   \begin{enumerate}
     \item  $\{Hath\}_{n}$ which is short for \textbf{Ha}rtog's \textbf{th}eorem in $\mathbb{C}^{n}$,
   \item  $\{Necs\}_{n}$ which is short for \textbf{N}on \textbf{e}xistence of \textbf{c}omplex \textbf{s}pheres $\mathcal{S}_{2n}$.
   \end{enumerate}
  \item It is clear that for $n \geq 2$, proposition $\left\{Hath\right\}_{n}$ is the statement for balls of Hartogs theorem, and for our
  purpose, it will be enough sufficient to restrict the statement of Hartogs' theorem to balls.
      \end{enumerate}
\end{remark}
\subsection{The main theorem : statement and strategy of proof.}
\subsubsection{\textbf{Statement of the main theorem.}}$ $\\
The ultimatum goal of this paper is then to prove the following main theorem.
\begin{theorem}\label{equiv}$ $ For all $n \geq 1$, the following equivalence holds:
$$\left\{Hath\right\}_{n}  \Longleftrightarrow \left\{Necs\right\}_{n}.$$
\end{theorem}
\begin{remark}$ $\\
\begin{enumerate}\label{remarque apres equiv}
\item Once theorem \ref{equiv} proved, proposition $\{Necs\}_{n}$ for $n \geq 2$, becomes then a geometric phenomenon equivalent to
the analytic phenomenon $\left\{Hath\right\}_{n}$.
 \item It is obvious that for $n=1$, the equivalence $\left\{Hath\right\}_{1}  \Longleftrightarrow \left\{Necs\right\}_{1}$ holds.
 Indeed, both propositions $\{Hath\}_{1}$ and $\{Necs\}_{1}$ are false, because for $D= \mathcal{B}_{1}$ and $\mathcal{K} = \left\{ 0\right\}$,
  the holomorphic function $f(z)= \frac{1}{z}$ defined on $\mathcal{B}_{1}-\{ 0\}$, does not admit an analytic continuation to $\mathcal{B}_{1}$,
   which means that $\{Hath\}_{1}$ is false, and in the same time, the sphere $\mathcal{S}_{2}$ is a Riemann surface, which means that $\mathcal{S}_{2}$
   admits a complex
   structure, and then $\{Necs\}_{1}$ is also false.
\item Observe that for $n \geq 2$, $\left\{Hath\right\}_{n}$ is already true, it is as mentioned above the statement for balls of the well known
 theorem on analytic continuation of holomorphic functions in $\mathbb{C}^{n}$, proved by Hartogs in 1906 (see \cite{10}). However, until proving
 the main theorem (theorem \ref{equiv}), we will not consider $\left\{Hath\right\}_{n}$  for $n \geq 2$, as already a theorem, but only as a logical
  proposition, without looking  whether it is true or false. But once we prove theorem \ref{equiv}, we will finally take into account the fact that
  for $n  \geq 2$, $\left\{Hath\right\}_{n}$ is already true, and then we will deduce immediately the exactness of $\left\{Necs\right\}_{n}$
  for $n \geq 2$, that, there exist no complex structures on the sphere $\mathcal{S}_{2n}$ for $n\geq 2$. Hence, the problem of complex spheres will
  finally be solved. In the same time, we obtain back a geometric interpretation of Hartogs' theorem in terms of non existence of complex
  structures on the spheres $\mathcal{S}_{2n}$ for $n \geq 2$. As far as I know, this is the first time that a geometric interpretation of Hartogs' theorem
  in terms of complex analytic geometry is given, since it is proved in 1906.
\end{enumerate}
\end{remark}
\subsubsection{\textbf{A version of $\left\{ Hath \right\}_{n}$ for $n  \geq 2$, in a complex chart}}$ $\\
Since it will be question to apply Hartogs' theorem on a sphere, it becomes then interesting to formulate a version of $\left\{Hath\right\}_{n}$
in a chart of
a complex analytic manifold $X$ of $dim_{\mathbb{C}}X \geq 2$. This is given in lemma \ref{Hartogs local} below.  But let's recall that we shall
continue temporarily (that is, before proving the main theorem) to consider  proposition $\left\{ Hath \right\}_{n}$ not as already a theorem, but only
as a logical proposition.
\begin{lemma}\label{Hartogs local}
Let $(U,\varphi)$ be a complex chart of a complex analytic manifold $X$ of $dim_{\mathbb{C}}X \geq 2$, and let $D$ be an open subset of $U$ and
 let $\mathcal{K}$ be a compact subset of $D$ such that:
 \begin{enumerate}
\item
 $D - \mathcal{K}$ is connected,
 \item $\varphi(D)$ is a ball of $\mathbb{C}^{n}$.
   \end{enumerate}
   Assume that proposition $\left\{ Hath \right\}_{n}$ is true, then for every holomorphic mapping $f: D -\mathcal{K}\longrightarrow
   \mathbb{C}^{k}$, there exists a holomorphic mapping $F: D \longrightarrow \mathbb{C}^{k}$ such that $F = f$ on $D - \mathcal{K}$.
\end{lemma}
\begin{proof}
Let $f: D - \mathcal{K} \longrightarrow \mathbb{C}^{k}$ be a holomorphic mapping. This means that the mapping $\widetilde{f}=
f \circ \varphi^{-1}: \varphi (D) - \varphi (\mathcal{K}) \longrightarrow \mathbb{C}^{k}$ (which is the expression of $f$ in the chart $(U,\varphi)$)
is holomorphic. Since $\varphi$ is a homeomorphism, then $\varphi(\mathcal{K})$ is a compact subset of $\varphi(D) \subset \mathbb{C}^{n}$ and $\varphi(D)
- \varphi (\mathcal{K})$ is connected in $\mathbb{C}^{n}$. $\varphi(D)$ being a ball of $\mathbb{C}^{n}$, then according to $\left\{Hath\right\}_{n}$,
there exists a holomorphic mapping $\widetilde{F}: \varphi(D) \longrightarrow \mathbb{C}^{k}$ such that $\widetilde{F} = \widetilde{f}$ in
$\varphi (D) - \varphi
(\mathcal{K})$. But $\widetilde{F}$ is exactly the expression in the chart $(U,\varphi)$ of the mapping $F = \widetilde{F} \circ \varphi:
 D \longrightarrow \mathbb{C}^{k}$ which is holomorphic in $D$ and satisfies $F = f$ in $D - \mathcal{K}$. Lemma \ref{Hartogs local} is then proved.
\end{proof}
\subsection{Hypothetic $\mathcal{A}-$submanifolds of $\mathcal{S}_{2n}$}$ $\\
As mentioned in the introduction, the main geometric object to be studied in this paper is the notion of $\mathcal{A}-$submanifold of $\mathcal{S}_{2n}$.
This notion is introduced in the definition \ref{A-sous variete de S_2n} below, and it will be a powerful tool to prove the main theorem \ref{equiv}.
But let's insist on the fact that the notion of $\mathcal{A}-$submanifold of $\mathcal{S}_{2n}$ is purely hypothetic, since
its existence depends only on the assumption that proposition $\left\{Necs\right\}_{n}$ is false, that is,
on the assumption that $\mathcal{S}_{2n}$ admits a complex structure $\mathcal{A}$, and then it will be used only as a technical tool to prove
by contradiction theorem \ref{equiv} .
\subsubsection{Notations and definition}\label{L'atlas A}$ $\\
Assume that $\left\{Necs\right\}_{n}$ is false, that is, $\mathcal{S}_{2n}$ admits a complex structure defined by an hypothetic complex atlas
$$
    \mathcal{A}  = \bigg\{ \left( U_{\jmath}, \varphi_{\jmath}\right), \phantom{rrr} \jmath \in \mathcal{J}  \bigg\}
$$
where $\left\{U_{\jmath},  \phantom{r}\jmath \in \mathcal{J} \right\}$ is an open covering of $\mathcal{S}_{2n}$, and
$\varphi_{\jmath}: U_{\jmath}\subset \mathcal{S}_{2n} \longrightarrow \varphi_{\jmath}\left( U_{\jmath}\right) \subset \mathbb{C}^{n}$ are homeomorphisms
such that for all $(\imath,\jmath) \in \mathcal{J}^{2}$ with $U_{\imath} \bigcap U_{\jmath} \neq\emptyset$, the following change of charts
$$
\varphi_{\jmath} \circ \varphi_{\imath}^{-1} :\varphi_{\imath} \left(U_{\imath} \bigcap U_{\jmath}\right)\longrightarrow \varphi_{\jmath} \left(U_{\imath}
 \bigcap U_{\jmath}\right)
$$
is biholomorphic.
By refining if necessarily the atlas $\mathcal{A}$, we can always assume taking into account the compactness of $\mathcal{S}_{2n}$, that $\mathcal{A}$ has
a finite number of charts
$$card (\mathcal{A}) = N +1 $$
 and that $\phantom{r}\mathcal{A}$ satisfies furthermore the following conditions:\\
 \begin{equation}\label{condition 1 des cartes}
   \textrm{For all}\phantom{r}\jmath \in \mathcal{J},\phantom{rrrr} \varphi_{\jmath}\left( U_{\jmath} \right) \phantom{r} \textrm{is a ball of} \phantom{r}
   \mathbb{C}^{n}.\phantom{rrrrrrrrrrrrrr}
   \end{equation}
   \begin{equation}\label{condition 2 des cartes}
    \textrm{For all} \phantom{r}(\imath,\jmath) \in \mathcal{J}^{2}, \phantom{rr} U_{\imath}\bigcap  \partial U_{\jmath}\phantom{rr}
     \textrm{is simply connected}.
 \end{equation}
 \begin{remark}
In all that follows, the hypothetic complex atlas $\mathcal{A}$ will be fixed once for all, and whenever we refer to an hypothetic complex structure
 on $\mathcal{S}_{2n}$ , this will always be the complex structure $\mathcal{A}$ considered above.
\end{remark}
\begin{definition}\label{A-sous variete de S_2n}$ $\\
Assume that proposition $\left\{Necs\right\}_{n}$ is false, and let then
$$
    \mathcal{A}  = \bigg\{ \left( U_{\jmath}, \varphi_{\jmath}\right), \phantom{rrr} \gamma \in \mathcal{J}  \bigg\}
$$
be the hypothetic complex structure on $\mathcal{S}_{2n}$ considered above. An hypothetic complex submanifold $X$ of $\mathcal{S}_{2n}$ is said to be
an $\mathcal{A}-$submanifold of $\mathcal{S}_{2n}$, if there exist a chart
$\left( U_{d}, \varphi_{d}\right) \in \mathcal{A}$, and a compact subset
$\mathcal{K}$ of $\mathcal{S}_{2n}$ such that the following conditions are satisfied:
\begin{itemize}
  \item [(a)]$X = \mathcal{S}_{2n} - \mathcal{K}$ \phantom{rrrrrrrrrrrr} ($X$ is then open).
  \item [(b)] $\mathcal{K}\subset U_{d}$.
  \item [(c)] $\varphi_{d}(\mathcal{K})= \mathcal{S}_{2n-2}(b,r)$ \phantom{rrrrrr} (that is, $\varphi_{d}(\mathcal{K})$ is a sphere of $\mathbb{C}^{n}$
   of $dim = 2n-2$).
\end{itemize}
$\left( U_{d}, \varphi_{d}\right)$ is said to be a distinguished chart of $\mathcal{A}$ for the $\mathcal{A}-$submanifold $X$.
\end{definition}
\begin{remark}\label{Phi(K) est une sphere centree en 0}$ $\\
\begin{enumerate}
  \item [(a)] By observing that the translations of $\mathbb{C}^{n}$ are biholomorphic mappings, then without changing the hypothetic complex structure
of $\mathcal{S}_{2n}$,
we can always  assume that $\varphi_{d}(\mathcal{K})$ is a sphere centered at $b = 0$, that is
\begin{equation}\label{phi(K)}
    \varphi_{d}\left( \mathcal{K}\right) = S_{2n-2}\left( 0, r_{2}\right).
\end{equation}
 \item [(b)] The distinguished chart $\left( U_{d}, \varphi_{d}\right) \in \mathcal{A}$ is not necessarily unique, but the following proposition shows
that for a given $\mathcal{A}-$submanifold $X$, we can always choose an hypothetic complex atlas $\widetilde{\mathcal{A}}$ of the sphere $\mathcal{S}_{2n}$
 finer
than $\mathcal{A}$ such that $X$ admits one and only one distinguished chart.
\end{enumerate}
\end{remark}
\begin{proposition}\label{proposition d'unicité de la carte distinguée}
Let $X$ be an $\mathcal{A}-$submanifold of $\mathcal{S}_{2n}$. Then there exists on $\mathcal{S}_{2n}$ an atlas $\widetilde{\mathcal{A}}$
 equal or finer than $\mathcal{A}$, such that $X$ viewed as $\widetilde{\mathcal{A}}-$submanifold of $\mathcal{S}_{2n}$ admits one and only one
  distinguished chart.
\end{proposition}
\begin{proof}$ $\\
Let $\left( U_{d}, \varphi_{d}\right)$ and $\left( U_{d'}, \varphi_{d'}\right)$ be two distinguished charts for the $\mathcal{A}-$submanifold $X$. $X$
can then be written as follows $$X = \mathcal{S}_{2n}- \mathcal{K} $$
where $\mathcal{K}$ is a compact subset of $\mathcal{S}_{2n}$ and  $$\mathcal{K} \subset U_{d} \bigcap U_{d'}.$$
Let $\mathbf{T}$ be an open neighborhood of $\mathcal{K}$  such that
$$
\mathcal{K} \subset \mathbf{T} \subset U_{d} \bigcap U_{d'},
$$
that is, such that $$\left(U_{d'} - \overline{\mathbf{T}}\right) \bigcap \mathcal{K} = \emptyset.$$
By recovering in $\varphi_{d'}\left( U_{d'} \right) \subset \mathbb{C}^{n}$, the following open set
$$\varphi_{d'}\left( U_{d'}\right) - \varphi_{d'}\left( \overline{\mathbf{T}}\right)$$ by small open balls $\mathcal{B}(a,r)$ such that
$$  \mathcal{B}(a,r) \bigcap \varphi_{d'}\left( \overline{\mathbf{T}}\right) = \emptyset,$$
we can then define new complex charts $\left( U_{d'}^{a,r}, \varphi_{d'}^{a,r}\right)$ by setting
\begin{equation*}
\left\{
    \begin{split}
       U_{d'}^{a,r}  & =   \varphi_{d'}^{-1}\left(\mathcal{B}(a,r)\right) \\
       \varphi_{d'}^{a,r}  & = \varphi_{d'}\big/U_{d'}^{a,r} \phantom{rrrrrr} \textrm{(that is, by restriction)}.
     \end{split}
     \right.
\end{equation*}
Since $\mathcal{S}_{2n}$ is compact, then by replacing in $\mathcal{A}$, the distinguished chart $\left( U_{d'}, \varphi_{d'}\right)$ by a
union of a finite number of new charts as defined above, we obtain a new atlas denoted $\widetilde{\mathcal{A}}$,
finer than $\mathcal{A}$, in which, the chart $\left( U_{d}, \varphi_{d}\right)$ is the only distinguished chart for
the $\widetilde{\mathcal{A}}-$submanifold $X$.
\end{proof}
\begin{remark}$ $\\
1) If there exist more than two distinguished charts for an $\mathcal{A}-$submanifold $X$, we prove
proposition \ref{proposition d'unicité de la carte distinguée} by induction. \\
2) By substituting if necessarily the atlas $\mathcal{A}$ by $\widetilde{\mathcal{A}}$, we can according
 to proposition \ref{proposition d'unicité de la carte distinguée}, always assume that a given $\mathcal{A}-$submanifold $X$ admits
 one and only one distinguished chart $\left( U_{d}, \varphi_{d}\right)$.
\end{remark}
\subsubsection{Complex structure of an $\mathcal{A}-$submanifold of $\mathcal{S}_{2n}$}\label{structure complexe de X}$ $\\
 Let $$ X = \mathcal{S}_{2n} - \mathcal{K}$$ be an $\mathcal{A}-$submanifold of $\mathcal{S}_{2n}$.
We want to describe the complex structure of $X$ inhered from $\mathcal{A}$. For this,  let the distinguished
chart $\left( U_{d}, \varphi_{d}\right) \in \mathcal{A}$ for $X$,
 and let by condition (\ref{condition 1 des cartes}) the ball
 $$
 \mathcal{B}\left( a , r_{1}\right) = \varphi_{d} \left( U_{d}\right) \subset \mathbb{C}^{n}
 $$
 and let by (\ref{phi(K)}) the sphere
 $$
 S_{2n-2}\left( b, r_{2}\right) = \varphi_{d}\left( \mathcal{K}\right) \subset \mathcal{B}\left( a , r_{1}\right).
 $$
 Since $dim \phantom{.} S_{2n-2}\left( 0, r_{2}\right) = 2n-2$, there exists then a real vector space $F \subset \mathbb{C}^{n}$
$$F = \bigg\{ \zeta \in \mathbb{C}^{n}, \phantom{rr} \left\langle w , \zeta \right\rangle =0\bigg \}$$
of $dim \phantom{.}F = 2n-1$, with $w \neq 0$,
 such that
 $$ S_{2n-2}\left( 0, r_{2}\right) \subset  F \subset \mathbb{C}^{n}.$$
Set now
\begin{equation}\label{couronne}
    \mathcal{N}\left(a,r_{1},r_{2}\right): = \mathcal{B}\left( a , r_{1} \right) - S_{2n-2}\left( 0, r_{2} \right) \subset \mathbb{C}^{n}.
\end{equation}
With these notations, the complex structure of the $\mathcal{A}-$submanifold $X$ is then given by the restricted atlas
$$\mathcal{A}\big/ X : = \bigg\{ \left( \widetilde{U_{\jmath}}, \widetilde{\varphi_{\jmath}}\right), \phantom{rrr} \jmath \in \mathcal{J}  \bigg\}$$
where the charts $\left( \widetilde{U}_{\jmath}, \widetilde{\varphi}_{\jmath} \right)$ of $\mathcal{A}\big/ X$ are defined as follows:
\begin{enumerate}
\item  For $ \jmath = d$, the chart $\left( \widetilde{U}_{d}, \widetilde{\varphi}_{d} \right)$ called distinguished chart for $X$  is given by
\begin{enumerate}
  \item $ \widetilde{U}_{d} : =  U_{d} - \mathcal{K} \subset X.$
  \item  $\widetilde{\varphi}_{d}: = \varphi_{d}\big/ _{\left(U_{d} - \mathcal{K}\right)}$ \phantom{rrrrrrrrrrrr} ($\widetilde{\varphi}_{d}$ is
   the restriction of $\varphi_{d}$ to $U_{d} - \mathcal{K}$).
\end{enumerate}
The restricted mapping
$$
\varphi_{d}: U_{d} - \mathcal{K} \longrightarrow  \mathcal{N}(a,r_{1},r_{2}).
$$
is then a homeomorphism.
\item  For all $\jmath\in\mathcal{J}-\left\{d\right\}$, $\phantom{rr} \left(\widetilde{U}_{\jmath},
 \widetilde{\varphi}_{\jmath}\right)  = \left(U_{\jmath}, \varphi_{\jmath}\right)$\\
    which implies that for all $\jmath \neq d$,  $\widetilde{\varphi}_{\jmath}
    \left( \widetilde{U}_{\jmath}\right) = \varphi_{\jmath} \left( U_{\jmath}\right)$ is a ball of $\mathbb{C}^{n}$.
\end{enumerate}
Let $\mathcal{W}$ be the relatively compact open set of $X$ defined by
$$
 \mathcal{W} = \mathcal{S}_{2n} - \overline{U_{d}}.
 $$
$\mathcal{W}$ is then covered by all domains of charts $ U_{\jmath}$ with $ \jmath \neq d$, that is
$$
\mathcal{W} \subset \bigcup_{\jmath\neq d} U_{\jmath}.
$$
Since by (\ref{condition 1 des cartes}) $\varphi_{d}\left( U_{d}\right)$ is an open ball of $\mathbb{C}^{n}$, we obtain
\begin{equation}\label{le bord de W est une sphere}
    \partial \mathcal{W} = \partial U_{d}\approx \mathcal{S}_{2n-1}
\end{equation}
and since by stereographic projection, $\varphi_{d}\left(\mathcal{W}\right))$ is relatively compact in $\mathbb{C}^{n}$,
 we deduce that
\begin{equation}\label{W est homeomorphe à une boule}
    \mathcal{W} \approx \mathcal{B}(0,r).
\end{equation}
\begin{remark} $ $\\The geometry of an $\mathcal{A}-$submanifold  $X$ of $\mathcal{S}_{2n}$ is characterized by two objects:
\begin{enumerate}
  \item The distinguished complex chart $\psi_{d}: U_{d} - \mathcal{K} \longrightarrow  \mathcal{N}(a,r_{1},r_{2})$.
\item The relatively compact open subset $\mathcal{W} = \mathcal{S}_{2n}- \overline{U_{d}}$, which is covered by the other charts
$\left(U_{\jmath}, \varphi_{\jmath}\right)$ and which is homeomorphic to an open ball of $\mathbb{C}^{n}$.
\end{enumerate}
In this paper, we will prove for $\mathcal{A}-$submanifolds of $\mathcal{S}_{2n}$, two fundamental properties:
\begin{enumerate}
  \item Every $\mathcal{A}-$submanifold of $\mathcal{S}_{2n}$ is homeomorphic to $\mathbb{C}^{n}$.
   \item Every $\mathcal{A}-$submanifold $X$ of $\mathcal{S}_{2n}$ is strictly pseudoconvex, and then by Cartan's theorem B, the group of
    cohomology of bidegree (0,1) of $X$ is trivial
       $$\mathcal{H}^{0,1}\left(X, \mathbb{C}^{n} \right) = \left\{ 0 \right\}
       .$$
\end{enumerate}
\end{remark}
\subsubsection{\textbf{The strategy of proof.}}$ $\\
It is now legitimate to ask how one can apply Hartogs' theorem to prove the theorem of complex spheres?
Or in other words, how one can prove the implication $\{Hath\}_{n} \Longrightarrow \left\{Necs\right\}$ for $n \geq 2$?\\
Before presenting our approach of proof, let's start with some remarks:
\begin{remark}
\begin{enumerate}
  \item It is obvious for reason of compactness, that $\mathcal{S}_{2n}$ can never admit a complex structure in one chart.
  \item Hartogs' theorem \ref{hartogs}, implies immediately that $\mathcal{S}_{2n}$ with $n \geq 2$, can neither admit a complex structure in two charts.\\
  Indeed, If $\left(U_{1}, \varphi_{1}\right)$ and $\left( U_{2}, \varphi_{2} \right)$ are two complex charts on $\mathcal{S}_{2n}$ with $$S_{2n}=U_{1}
  \bigcup U_{2}$$ then the holomorphic mapping
$$\varphi_{2}: U_{1}\cap U_{2}\subset \mathcal{S}_{2n} \longrightarrow \varphi_{2}\left(U_{1}\cap U_{2} \right)\subset
\mathbb{C}^{n}$$ ($U_{1}\bigcap U_{2}$
connected) admits by Hartogs' theorem an analytic continuation to $U_{1}$. Therefore, $\varphi_{2}$ admits an analytic continuation to the whole sphere
$\mathcal{S}_{2n}$, and then becomes constant. For the same reason, we obtain the same result for $\varphi_{1}$. But this contradicts the fact that
$\varphi_{2}\circ \varphi_{1}^{-1}$ is bi-holomorphic.
\end{enumerate}
\end{remark}
The proof that $\mathcal{S}_{2n}$ with $n\geq 2$, could never admit a complex structure in more than two charts, seems somewhat more complicated, but it is
precisely this hypothetic situation which we will show in this paper that it is impossible.\\
\textbf{What is the main idea of proof?} $ $\\ This idea consists roughly speaking, in proving by contradiction that if for $n \geq 2$,
$ \left\{Necs\right\}$ is false, then for an appropriate atlas $\mathcal{F}=
 \bigg\{\left(W_{\gamma},\Phi_{\gamma} \right),\phantom{rr} \gamma\in \mathcal{L}\bigg\}$ finer than $\mathcal{A}$, one can construct for all pairing
  of different charts $\left(W_{\alpha},\Phi_{\alpha} \right)$,  $\left(W_{\beta},\Phi_{\beta} \right) \in \mathcal{F}$ satisfying the condition
  $W_{\alpha}\bigcap W_{\beta} \neq \emptyset$,
  an $\mathcal{A}-$submanifold
   $X_{\alpha,\beta}$ of $\mathcal{S}_{2n}$ in the sense of definition \ref{A-sous variete de S_2n}, and a harmonic mapping
$$H: X_{\alpha,\beta} \longrightarrow \mathbb{C}^{n}$$ with $H = H_{1}+ \overline{H_{2}}$,  ($H_{1}$ and $H_{2}$ are holomorphic)
such that for all $\alpha \in \mathcal{L}$
$$\partial H/W_{\alpha} = \partial\Phi_{\alpha}$$
which implies by Hartogs' theorem that $\Phi_{\alpha}$ is constant. The proof of the existence of such harmonic mapping $H$ needs some preparation, and it
will be done later in section 4). Leaving then the details to section 4), we sketch below just the procedure of proof.\\
\textbf{\underline{The procedure of proof.}} $ $
 \begin{enumerate}
\item We begin by proving that $X_{\alpha,\beta}$ is homeomorphic to $\mathbb{C}^{n}$. This result is given in section 2).
\item Then we prove that $\mathcal{H}^{0,1}\left(X_{\alpha,\beta},\mathbb{C}^{n} \right) = \left\{ 0 \right\}$, which implies by Dolbeault's resolution of
the sheaf $\mathcal{O}^{n}$ that the first $\check{C}$ech cohomology group is trivial $$\mathcal{H}^{1}\left(X_{\alpha,\beta}, \mathcal{O}^{n}\right)
= \left\{ 0 \right\}.$$
    The proof of the triviality of the group $\mathcal{H}^{1}\left(X_{\alpha,\beta}, \mathbb{C}^{n}\right)$ will be the main object of section 3).
 \item  We deduce from 2) that the 1-cycle $g_{\mu,\nu}$ with coefficients in the sheaf $\mathcal{O}^{n}$ defined on $X_{\alpha,\beta}$ by
 \begin{displaymath}
 f_{\mu,\nu} = \left\{
 \begin{split}
   \Phi_{\beta} - \Phi_{\alpha}& = \phantom{rrr} \textrm{if}\phantom{r} (\mu,\nu)=(\alpha,\beta) \\
   0 \phantom{rrr}& =  \phantom{rrr} \textrm{if}\phantom{r} (\mu,\nu) \neq (\alpha,\beta).
 \end{split}
 \right.
 \end{displaymath}
 is a 1-coboundary, that is  $$f_{\mu,\nu} = F_{\nu} - F_{\mu}.$$
 \item  The restrictions $F_{\alpha}/(X_{\alpha,\beta}-\overline{W_{\alpha}})$ and $F_{\beta}/(X_{\alpha,\beta}-\overline{W_{\beta}})$
 being holomorphic, admit both by Hartogs' theorem (\footnote{It is at this stage of the proof that we will use the fact that $ n \geq 2$.})
 an analytic continuation to the sphere $\mathcal{S}_{2n}$, and therefore become constant on $\mathcal{S}_{2n}$, and then
     $$\partial\Phi_{\alpha} = \partial \Phi_{\beta} \phantom{rrrrrr} \textrm{on} \phantom{rr}W_{\alpha}\cap W_{\beta}.$$
 Gluing the different pieces $\partial\Phi_{\gamma} $ together, there exists then a $\partial-$closed form $\omega$ with values in $\mathbb{C}^{n}$, of
  bidegree $(1,0)$, such that $\omega/W_{\gamma} = \partial \Phi_{\gamma}$
  \item Once again, using the triviality of the group $\mathcal{H}^{1,0}\left( X_{\alpha,\beta} ,\mathbb{C}^{n} \right) $, we prove that there exists a
  harmonic mapping $H = H_{1} + \overline{H}_{2}: X_{\alpha,\beta} \longrightarrow \mathbb{C}^{n}$ (with $H_{1}$ and $H_{2}$ holomorphic), such that
   $\partial H = \omega.$
  \item  $H_{1}$ and $H_{2}$ admit by Hartogs' theorem analytic continuations to the sphere $S_{2n}$, and then become constant. Hence, by (3) and (4)
  we obtain $\partial \Phi_{\gamma} = 0$. But this contradicts the fact that $(W_{\gamma},\Phi_{\gamma})$ is a complex analytic chart.
 \end{enumerate}
 \begin{remark}$ $\\ 1) It appears clearly that the main piece which foods into the procedure of proof sketched above, is  the fact that the group of
  cohomology of Dolbeault of bidgree (0,1) of any $\mathcal{A}-$submanifold of $\mathcal{S}_{2n}$ is trivial, and then the reader who would admit
  temporarily this result could read directly the proof of the main theorem \ref{equiv} in section 4).\\
2) As mentioned in the introduction, although the problem of complex spheres stands only for $\mathcal{S}_{6}$, our approach is however valid for all
spheres $\mathcal{S}_{2n}$ with $ n \geq 2$, and then, we don't have to use the result of Borel Serre \cite{1}.
\end{remark}
\section{\textbf{Topology of the $\mathcal{A}-$submanifolds of $S_{2n}$}}
This section is devoted to prove two topological results: The first one is that every $\mathcal{A}-$submanifold of $\mathcal{S}_{2n}$
is homeomorphic to $\mathbb{C}^{n}$. The second one (see proposition \ref{proposition des boules}) shows that under some supplementary conditions,
If $A$ and $B$  are open sets of $\mathbb{C}^{n}$, both homeomorphic to balls of $\mathbb{C}^{n}$, then $B - A$ is homeomorphic to a ball
 of $\mathbb{C}^{n}$. As a consequence of this result, we deduce for the relatively compact open set defined by
 $\mathcal{W} = \mathcal{S}_{2n}- \overline{U_{d}}$ where $\left( U_{d}, \varphi_{d}\right)$ is the distinguished chart
  of an $\mathcal{A}-$submanifold of $\mathcal{S}_{2n}$, that there exist open sets
  $$\mathcal{W}_{1} , \mathcal{W}_{2}, ..., \mathcal{W}_{N}$$ all homeomorphic to balls of $\mathbb{C}^{n}$ such that
$$
\mathcal{W}_{1} \subset \mathcal{W}_{2} \subset ...\subset \mathcal{W}_{N}= \mathcal{W}.
$$
Such a result will be very useful in the next section to prove that any $\mathcal{A}-$submanifolds of $\mathcal{S}_{2n}$ is strictly pseudoconvex.
To prove these results, we need first to check some lemmas concerning some special homeomorphisms of $\mathbb{C}^{n}$.
\subsection{Some special homeomorphisms of $\mathbb{C}^{n}$}
\begin{lemma}\label{preliminaire à la proposition}$ $\\
Let the ball $\mathcal{B}_{r} \subset \mathbb{C}^{n}$, and let $\mathcal{K}$ be a compact subset of $\mathcal{B}_{r}$. Then for every $ \varepsilon > 0$,
there exists a homeomorphism $f: \mathbb{C}^{n}  \longrightarrow   \mathbb{C}^{n}$ satisfying the following conditions:
\begin{equation*}\left\{
    \begin{split}
      f(z)  &  = z \phantom{rr} \textrm{if} \phantom{rr}z \not\in \overline{\mathcal{B}_{r}} \\
        diam\big(f(\mathcal{K})\big) &  < \varepsilon.
    \end{split}
    \right.
\end{equation*}
\end{lemma}
\begin{proof}
Since the ball $\mathcal{B}_{r}$ is open and contains the compact $\mathcal{K}$, then there exists an open ball $\mathcal{B}_{r'}$ with $(0<  r' < r)$,
such that
$$
\mathcal{K} \subset \mathcal{B}_{r'} \subset \mathcal{B}_{r}.
$$
Let the piecewise linear homeomorphism $\chi: \mathbb{R}_{+}\longrightarrow \mathbb{R}_{+}$ given by:
\begin{equation}\label{le premier ki}
  \chi(t): =   \left\{
\begin{split}
   &  \phantom{rrrrrr}\frac{\varepsilon.t}{2 r'}\phantom{rrrrrrrrrrrrrrrrr} \textrm{ if } \phantom{r}  0 \leq t \leq r'\\ & \\
   &   \frac{ \left(r-\frac{\varepsilon}{2}\right)(t- r')}{r-r'} + \frac{\varepsilon}{2}\phantom{rrrrrrr} \textrm{ if } \phantom{r}  r' \leq t \leq r\\
   & \\
   & \phantom{rrrrrr} t \phantom{rrrrrrrrrrrrrrrrrr} \textrm{ if } \phantom{r} t \geq r.
\end{split}
\right.
\end{equation}
and define the mapping $f: \mathbb{C}^{n} \longrightarrow \mathbb{C}^{n}$ by the following expression:
\begin{equation}\label{premier homeomorphism}
    f(z) : =
    \left\{
    \begin{split}
       & z.\frac{\chi(\|z\|)}{\|z\|} \phantom{rrrrrrrrr} \textrm{if} \phantom{rr} z \neq 0 \\
        & \phantom{rrrr} 0  \phantom{rrrrrrrrrrrr} \textrm{if} \phantom{rr} z = 0.
    \end{split}
    \right.
\end{equation}
By writing $z \in \mathbb{C}^{n}$ in spherical coordinates, we observe that the homeomorphism $\chi : \mathbb{R}_{+}\longrightarrow \mathbb{R}_{+}$
acts only on the norm $\| z\|$, and not on the angles of $z$.  The mapping $f$ is then a homeomorphism of $\mathbb{C}^{n}$. Since for $\|z\| \geq 1$,
 we have $\chi(\|z\|) = \|z\|$, we deduce then that $f(z)=z$ in $\mathbb{C}^{n}-\mathcal{B}_{r}$. Now let $z, \xi \in \mathcal{K}$,
 that is, $\|z\|\leq r'$ and $\|\xi\| \leq r'$. \\
1) If $\xi = 0$, then  $\left\|f(z) - f(0) \right\| = \left\| f(z)\right\| = \frac{\varepsilon .\left\|z\right\|}{2 r'} \leq \frac{\varepsilon}{2}
 \leq \varepsilon.$ \\
2) If both $z$ and $\xi$ are $\neq 0$, then
an elementary calculation gives
\begin{equation}\label{premier homeo}
    \begin{split}
      \| f(z)-f(\xi)\| & \leq \| f(z)\| + \|f(\xi)\|\\
& \leq \frac{\|z\|}{2\varepsilon r'} +  \frac{\|\xi\|}{2\varepsilon r'}\\
&\leq \frac{\varepsilon}{2} + \frac{\varepsilon}{2}\\
& \leq \varepsilon
    \end{split}
\end{equation}
which means that $diam(f(\mathcal{K}))\leq \varepsilon$. The proof of lemma \ref{preliminaire à la proposition} is then complete.
\end{proof}
 \begin{lemma}\label{tailleur}(\textbf{Flattening homeomorphism})$ $\\
 Let $\overline{\mathcal{B}}(a,r)$ be a closed ball of $\mathbb{C}^{n}$, and let $\mathcal{K}$ be a closed subset of $\overline{\mathcal{B}}(a,r)$
 such that the center $a \not\in \mathcal{K}$. Then for all $\varepsilon \in \left] 0,r  \right[$, there exists a homeomorphism $f: \mathbb{C}^{n}
 \longrightarrow \mathbb{C}^{n}$ satisfying the conditions:
 \begin{equation*}\left\{
  \begin{split}
    f(z) & = z \phantom{rrrrrr} \textrm{if} \phantom{r} z \not\in \overline{\mathcal{B}}(a,r) \\
     f (\mathcal{K}) & \subset \overline{\mathcal{B}}(a,r) - \mathcal{B}(a,r-\varepsilon).
  \end{split}\right.
 \end{equation*}
 \end{lemma}
 \begin{remark}$ $\\
 We don't know if the term "flattening homeomorphism" exists already in mathematics, but we adopt it here in order to give an exact idea on
 how the
 homeomorphism considered above acts. Indeed, $f$ acts really by "flattening" the compact $\mathcal{K}$ into the annulus
  $\overline{\mathcal{B}}(a,r) - \mathcal{B}(a,r-\varepsilon).$
 \end{remark}
 \begin{proof} Let $\varepsilon \in \left] 0, r \right[$ . Since the center of the ball $ a \not\in \mathcal{K}$, then
 $$
 b:= \inf_{z\in \mathcal{K}} \| z-a \| > 0.
 $$
 Consider the piecewise linear homeomorphism $\chi : \mathbb{R}_{+} \longrightarrow \mathbb{R}_{+}$ given by
 \begin{equation}\label{le deuxieme ki}
  \chi (t) : = \left\{
    \begin{split}
      &  \phantom{rr} \frac{t(r-\varepsilon)}{b}\phantom{rrrrrrrrrrrrr} \textrm{if} \phantom{rr} 0 \leq t \leq b\\
        & \frac{\varepsilon(t-b)}{(r-b)} + r-\varepsilon \phantom{rrrrrrrr} \textrm{if} \phantom{rr} b \leq t \leq r\\
        & \phantom{rrrr} t \phantom{rrrrrrrrrrrrrrrrrr} \textrm{if} \phantom{rr} r \leq t.
    \end{split}
    \right.
 \end{equation}
and define the mapping $ f: \mathbb{C}^{n} \longrightarrow \mathbb{C}^{n}$  by
\begin{equation}\label{deuxieme homeo}
    f(z): = \left\{
    \begin{split}
       &  a + \frac{\chi (\|z-a \|)}{\|z-a \|}. (z-a ) \phantom{rrrrrrrrrrrr} \textrm{if} \phantom{r} z \neq a\\
        &  \phantom{rrrr} a \phantom{rrrrrrrrrrrrrrrrrrrrrrrrrrrr} \textrm{if} \phantom{r} z = a.
    \end{split}
    \right.
\end{equation}
As in the proof of the previous lemma \ref{preliminaire à la proposition}, we observe by writing $z \in \mathbb{C}^{n}$ in spherical coordinates,
that the homeomorphism $\chi : \mathbb{R}_{+}\longrightarrow \mathbb{R}_{+}$ defined by (\ref{le deuxieme ki}), acts only on the norm $\| z\|$,
and not on the angles of $z$.  The mapping $f$ is then a homeomorphism of $\mathbb{C}^{n}$. It is clear from the expression (\ref{deuxieme homeo})
 that $f$ satisfies the required conditions, that is
  \begin{equation*}\left\{
  \begin{split}
    f(z) & = z \phantom{rrrrrr} \textrm{if} \phantom{r} z \not\in \overline{\mathcal{B}}(a,r) \\
     f (\mathcal{K}) & \subset \overline{\mathcal{B}}(a,r) - \mathcal{B}(a,r-\varepsilon).
  \end{split}\right.
 \end{equation*}
 \end{proof}
\begin{lemma}\label{lemme d'homothetie}$ $\\
Let the open ball $\mathcal{B}_{r}$ of $\mathbb{C}^{n}$, and let $\mathcal{M}_{0}$ be an open subset of $\mathbb{C}^{n}$ satisfying the following
 conditions:
\begin{enumerate}
  \item $ \mathcal{M}_{0} \bigcap \mathcal{B}_{r} \neq \emptyset.$
  \item $ \mathcal{B}_{r} \bigcap \partial \mathcal{M}_{0} $ is simply connected.
   \item 0 $\not\in \overline{\mathcal{M}_{0}}$.
\item  For all $z \in \mathcal{B}_{r}
\bigcap \partial \mathcal{M}_{0} $, we have
 $$\left[ 0 , z\right[ \bigcap \mathcal{B}_{r} \bigcap \partial \mathcal{M}_{0} = \left\{ z\right\}$$ where $\left[ 0 , z\right[$
  denotes the half-line starting from the center $0$ of $\mathcal{B}_{r}$ and containing $z$.
 \end{enumerate} For every point $ z \in \mathcal{B}_{r} \bigcap \partial \mathcal{M}_{0}$, define the point $s(z)$ by
 $$s(z) = \left[ 0 , z\right[ \bigcap \mathcal{S}_{2n-1}(0,r)$$ and let the closed cone
$$ \widehat{\mathcal{M}}: = \bigcup _{z\in  \overline{\mathcal{B}}_{r}\bigcap \partial \mathcal{M}_{0} } \left[0,z \right[.
$$
Then the mapping
$$
\mathcal{H}_{0}:\left( \mathcal{B}_{r} - \mathcal{M}_{0}\right) \longrightarrow \mathcal{B}_{r}
$$
defined by :
\begin{equation}\label{Homothetie}
    \left\{
    \begin{split}
      \mathcal{H}_{0}(z) & = \phantom{r} z \phantom{rrrrrrrrrr} \textrm{if}\phantom{r} z \not \in  \widehat{\mathcal{M}}\\
       \mathcal{H}_{0}(z) & = r.\frac{z}{\left\|z  \right\|} \phantom{rrrrrrr} \textrm{if}\phantom{r} z \in  \widehat{\mathcal{M}}
    \end{split}
    \right.
\end{equation}
is a homeomorphism.
\end{lemma}
 \begin{proof}
 We leave to the reader to check lemma \ref{lemme d'homothetie}.
 \end{proof}
\subsection{Applications to the topology of the $\mathcal{A}-$submanifolds of $\mathcal{S}_{2n}$}
 \subsubsection{Quotient topology}$ $\\
Let $E$ be a metric space, and let $\mathcal{K}$ be a compact subset of $E$. We define then on $E$ a relation of equivalence $\sim$ by
\begin{equation}\label{relation d'equivalence}
 z \sim z' \Longleftrightarrow \left\{
 \begin{array}{c}
   z = z'  \\
   \textrm{or} \\
  z  \phantom{r} \textrm{and} \phantom{r} z' \in \mathcal{K}
 \end{array}
 \right.
\end{equation}
and we note by $\pi: E \longrightarrow E/ \mathcal{K}$ the canonical projection. We want to prove that under some conditions on $\mathcal{K}$,
the quotient space $E/ \mathcal{K}$ is homeomorphic to $E$. This will be the object of proposition \ref{proposition fondamentale} below,
and then we deduce that every $\mathcal{A}-$submanifold of $\mathcal{S}_{2n}$ is homeomorphic to $\mathbb{C}^{n}$.
\begin{proposition}\label{proposition fondamentale}$ $\\
Let $(E,d)$ be a metric space, and let $\mathcal{K}$ be a compact subset of $(E,d)$, and consider on $E$ the relation of equivalence $\sim$
defined by
(\ref{relation d'equivalence}). Assume that the compact $\mathcal{K}$ admits in $(E,d)$ a fundamental
system of open neighborhoods $\left\{U_{k}\right\}_{k \in \mathbb{N}^{\ast}}$, such that each $U_{k}$ is relatively compact and is homeomorphic
to $\mathbb{C}^{n}$. Then for every open neighborhood $U$ of $\mathcal{K}$ which is relatively compact, there exists a
homeomorphism $h: E/\mathcal{\mathcal{K}} \longrightarrow E$ such that $h \circ \pi $ coincides with $ \mathbb{I}_{E}$ in $E - U$.
\end{proposition}
\begin{proof}$ $\\
Let us first begin by the following remark.
\begin{remark}
\emph{If the compact $\mathcal{K}$ is a singleton set, that is $\mathcal{K} = \left\{z_{0}\right\}$, then the equivalence
relation $"\sim"$ coincides with
the equality relation " =" , and then $E/\mathcal{K} = E $. In this case, we have nothing to prove, because the identity mapping
$h = \mathbb{I}_{E}$ answers the proposition. Therefore, we can assume that $\mathcal{K}$ is not a singleton set.}
\end{remark}
Since by hypothesis, the open sets $U_{k}$ are relatively compact, that is, each $\overline{U_{k+1}}$ is compact in $ U_{k}$, and
since they constitute a fundamental system of neighborhoods of $\mathcal{K}$, then
$$ \mathcal{K} = \bigcap _{k=1}^{\infty} U_{k}.$$
Without loss of generality, we can suppose $U = U_{1}$. We claim that, there exists a sequence of homeomorphisms $g_{k}: E \longrightarrow E$, such that:
\begin{enumerate}
\item $g_{1} = \mathbb{I}_{E}$.\\
\item For all $ k > 1$, $g_{k}$ satisfies the conditions:
\begin{equation}\label{condition d'homeomorphism}
\left\{
    \begin{split}
      g_{k}(z) & = g_{k-1}(z) \phantom{rrrrrr} \textrm{for all } \phantom{r}z\in  E-U_{k} \\
       diam \phantom{.} g_{k}\left( U _{k+1}\right) & < \frac{1}{k}.
    \end{split} \right.
    \end{equation}
\end{enumerate}
Indeed, suppose by induction, that we have already constructed the homeomorphism $g_{k-1}$ satisfying condition (\ref{condition d'homeomorphism}).
Since $g_{k-1}$ is continuous (and then uniformly continuous on every compact subset of $E$), then for every fixed compact $C \subset E$ and for every
subset $Y \subset C$, we have
\begin{equation}\label{equation de continuité}
\lim_{diam(Y)\longrightarrow 0}diam \left(g_{k-1}(Y)\right) = 0.
\end{equation}
That is, for any subset $Y \subset C$, and for all $k \in \mathbb{N}^{\ast}$, there exists a number $\delta_{k} > 0$ such that,
\begin{equation}\label{equation de continuité bis}
diam(Y) < \delta_{k} \Longrightarrow diam\left(g_{k-1}(Y)\right) < \frac{1}{k}.
\end{equation}
To construct the homeomorphism $g_{k}: E \longrightarrow E$ satisfying (\ref{condition d'homeomorphism}), consider the relatively compact
neighborhoods $U_{k}$ and $U_{k+1}$ of $\mathcal{K}$, and recall that there exists by hypothesis, a homeomorphism $\varphi_{k}: U_{k}\longrightarrow
\mathbb{C}^{n}$. \\Since $\overline{U_{k+1}}$ is compact in $U_{k}$, then $\varphi_{k}\left( \overline{U_{k+1}} \right)$ is compact in $\mathbb{C}^{n}$.
 Now let $\mathcal{B}_{r_{k}}$ be an open ball in $\mathbb{C}^{n}$ centered at 0 and containing the compact $\varphi_{k}\left(\overline{U_{k+1}}\right)$,
  then
 by lemma \ref{preliminaire à la proposition}, there exists a homeomorphism $f_{k}: \mathbb{C}^{n} \longrightarrow \mathbb{C}^{n}$ such that
\begin{equation*}\left\{
    \begin{split}
      f_{k}\left(\varphi_{k}(z ) \right)  &  = \varphi_{k}(z) \phantom{rr} \textrm{if} \phantom{rr}\varphi_{k}(z) \not\in \overline{\mathcal{B}_{r_{k}}}
      \\
        diam\big(f_{k}\left(  \varphi_{k}\left( \overline{U_{k+1}} \right)\right)\big) &  < \delta_{k}.
    \end{split}
    \right.
\end{equation*}
According to the commutative diagram
\begin{displaymath}
\xymatrix {U_{k} \ar[d]_{\varphi_{k}} \ar[r]^{\widetilde{g_{k}}}& U_{k} \ar[d]^{\varphi_{k}}  \\
\mathbb{C}^{n} \ar[r]_{f_{k}} & \mathbb{C}^{n}}
\end{displaymath}
there exists a homeomorphism $\widetilde{g_{k}} : U_{k} \longrightarrow U_{k} $ satisfying the conditions:
\begin{equation*}\left\{
    \begin{split}
      \widetilde{g_{k}}(z )   &  = z \phantom{rr} \textrm{if} \phantom{rr} z \not\in \varphi_{k}^{-1}\left(\overline{\mathcal{B}}_{r_{k}}\right) \\
        diam\big(\widetilde{g_{k}}\left( \overline{U_{k+1}} \right)\big) &  < \delta_{k}.
    \end{split}
    \right.
\end{equation*}
Taking into account that we have already $\widetilde{g_{k}}(z) = z$ on $U_{k}- \varphi_{k}^{-1}\left(\overline{\mathcal{B}}_{r_{k}}\right)$,
we can then extend the local homeomorphism $ \widetilde{g_{k}}: U_{k} \longrightarrow U_{k}$ to a global homeomorphism $ \widetilde{g_{k}}:
E \longrightarrow E$, by setting:
 $$
 \widetilde{g_{k}}(z)  = z \phantom{rrrr}  \textrm{if} \phantom{r} z \in E- U_{k}.
$$
Hence, the global homeomorphism $\widetilde{g_{k}}: E \longrightarrow E $ satisfies the conditions:
 \begin{equation*}\left\{
    \begin{split}
      \widetilde{g_{k}}(z) & = z \phantom{rrrr}  \textrm{if} \phantom{r} z \in E- U_{k}\\
      diam\big(\widetilde{g_{k}}\left( \overline{U_{k+1}} \right)\big) &  < \delta_{k}.
    \end{split}
    \right.
\end{equation*}
Define now the required mapping $g_{k}: E \longrightarrow E \phantom{rr}$ by setting:
$$
g_{k} = g_{k-1} \circ \widetilde{ g_{k}}.
$$
We claim that:
\begin{enumerate}
\item  $g_{k} : E  \longrightarrow E$ is a homeomorphism,\\
\item $g_{k} $  satisfies condition (\ref{condition d'homeomorphism}).
\end{enumerate}
Indeed, $ g_{k}$ is a homeomorphism because both $g_{k-1}$ and $\widetilde{ g_{k}}$ are homeomorphisms.\\
Since for all $z \in E- U_{k}$, we have $\widetilde{g_{k}}(z) = z$,
we obtain then
 \begin{equation}\label{premiere condition}
    g_{k}(z) =  g_{k-1}  \left( \widetilde{ g_{k}}(z) \right) =  g_{k-1}(z)\phantom{rrrrrr} \textrm{on}\phantom{rr}  E- U_{k}.
 \end{equation}
Furthermore, since $diam\big(\widetilde{g_{k}}\left( \overline{U_{k+1}} \right)\big) < \delta_{k}$,
it follows by (\ref{equation de continuité bis}) that
\begin{equation}\label{deuxieme condition}
    diam \left( g_{k}\left(U_{k+1} \right) \right) = diam\left( g_{k-1}\left( \widetilde{g_{k}}\left( U_{k+1} \right) \right)\right) < \frac{1}{k}.
\end{equation}
Hence condition (\ref{condition d'homeomorphism}) is satisfied. \\
Consider now the sequence of homeomorphisms $\left\{ g_{k}\right\}_{k \geq 1}$ satisfying condition (\ref{condition d'homeomorphism}).
We claim that $\left\{ g_{k}\right\}_{k \geq 1}$ admits a subsequence $\left\{ g_{k_{j}} \right\}_{j \geq 1}$ which converges point wise
 to a
continuous mapping $g: E \longrightarrow E$.\\ Indeed, fix a point $z_{0}\in \mathcal{K}$, and consider the sequence of points $\left\{ g_{k}(z_{0})
\right\}_{k \geq 1}$ of the metric space $E$. Since the compact $\mathcal{K}$ is assumed to be non reduced to a singleton set, then by
 condition (\ref{condition d'homeomorphism}), we obtain
$$\left\{g_{k}(z_{0})\right\}_{k\geq1} \subset \mathcal{K} = \bigcap _{k=1}^{\infty}U_{k}.$$
By the well known Bolzano-Weirstrass theorem,  the sequence of points $\left\{g_{k}(z_{0})\right\}_{k\geq1}$ admits then a
subsequence $\left\{ g_{k_{j}}(z_{0}) \right\}_{j \geq 1}$ which converges to a point $\displaystyle \ell \in \mathcal{K} =
\bigcap _{k=1}^{\infty}U_{k}$.\\
Consider now the mapping $g : E \longrightarrow E$ defined as follows:
\begin{equation*}\left\{
    \begin{split}
      g(z) & = g_{k_{j}}(z) \phantom{rrrr}  \textrm{if} \phantom{r} z \in E- U_{k_{j}}\\
       g(z) & = \ell \phantom{rrrrrrrr}  \textrm{if} \phantom{r} z \in \mathcal{K}.
    \end{split}
    \right.
\end{equation*}
By condition (\ref{condition d'homeomorphism}), the mapping $g$ is well defined. It is clear that $g$ is continuous in $E - \mathcal{K}$,
because it coincides with one of the homeomorphisms $g_{k_{j}}$. We claim that $g$ is also continuous at every point $z_{0} \in \mathcal{K}$.\\
 Indeed, Let the number $\varepsilon > 0$. Since $\displaystyle\lim_{z\longrightarrow z_{0}}g_{k_{j}}(z_{0}) = \ell$, there exists
  then $N_{0} \in \mathbb{N}$
 such that for all $ k_{j} \geq N_{0}$, we have $d\left( g_{k_{j}}(z_{0}) , \ell \right) \leq \frac{\varepsilon}{2}.$\\
 In another hand,  there exists a neighborhood $U_{_{N_{1}}}$ of $\mathcal{K}$, such that
 $$
 z \in U_{_{N_{1}}} \Longrightarrow d\left( g_{k_{j}}(z), g_{k_{j}}\left(z_{0}\right) \right) \leq \frac{\varepsilon}{2}.
 $$
 Now, let
 $N = max\left\{ N_{0} , N_{1}\right\}$, and let $z \in U_{_{N}} \subset U_{_{N_{1}}}$. \\
 1) If $z \in \mathcal{K}$, we have then $d(g(z), g(z_{0})) = 0$. \\
 2) If $ z \in \left( U_{_{N}} - \mathcal{K}\right)$, there exists $ k_{j} \geq N $, such that $g(z) = g_{k_{j}}(z) $.
  In this case, we obtain:
 \begin{equation*}
    \begin{split}
 d\left(g(z), g\left(z_{0}\right)\right) & = d\left(g(z), \ell\right) \\
    & = d\left(g_{k_{j}}(z), \ell\right) \\
    & \leq d\left(g_{k_{j}}(z), g_{k_{j}}(z_{0})\right) + d\left(g_{k_{j}}(z_{0}), \ell\right)\\
    & \leq \frac{\varepsilon}{2} + \frac{\varepsilon}{2} \\
    & = \varepsilon.
     \end{split}
\end{equation*}
Then $g$ is a continuous mapping on $E$, but since $\mathcal{K}$ is not a singleton set, $g$ is not a homeomorphism because
 $diam \left(g(\mathcal{K})\right) = 0.$\\
Now let $\pi : X \longrightarrow E/\mathcal{K}$ be the canonical projection, and let the relatively compact neighborhood $U = U_{1}$ of
$\mathcal{K}$. Since the mapping $g: E \longrightarrow E$ is constant in the compact $\mathcal{K}$, and $g_{k_{j}}$ are homeomorphisms,
 then $g$ factorizes as follows
$$g = h \circ \pi $$
 where $h: E/\mathcal{K} \longrightarrow E$ is a homeomorphism such that the following diagram commutes
\begin{displaymath}
\xymatrix {E \ar[d]_{\pi} \ar[dr]^{g} \\
E/\mathcal{K}\ar[r]_{h}& E }
\end{displaymath}
This means that the homeomorphism $h: E/\mathcal{K} \longrightarrow E$ is such that $h \circ \pi = g = \mathbb{I}_{E}$ in $E - U$. The proof of
 proposition \ref{proposition fondamentale} is then complete.
\end{proof}
\begin{corollary}\label{corollaire de topologie}$ $\\
If $\mathcal{K}$ is a compact subset of $\mathcal{S}_{2n}$, then $\mathcal{S}_{2n}- \mathcal{K}$ is homeomorphic to $\mathbb{C}^{n}$.
\end{corollary}
\begin{proof}
Indeed, since by stereographic projection, the compact $\mathcal{K}$ admits a fundamental system of relatively compact neighborhoods each of them
homeomorphic to $\mathbb{C}^{n}$, we deduce then according to proposition \ref{proposition fondamentale}, that for every neighborhood  $U$
of $\mathcal{K}$, there exists a homeomorphism
$$h: \mathcal{S}_{2n}/\mathcal{K} \longrightarrow  \mathcal{S}_{2n}$$ such that the following diagram commutes
\begin{displaymath}
\xymatrix {\mathcal{S}_{2n} \ar[d]_{\pi} \ar[dr]^{g} \\
\mathcal{S}_{n}/\mathcal{K}\ar[r]_{h}& \mathcal{S}_{2n} }
\end{displaymath}
and that $g = h o \pi $  coincides with $\mathbb{I}_{\mathcal{S}_{2n}}$ in $U$.\\
Since for all $z \in \left(\mathcal{S}_{2n}- \mathcal{K}\right)$, we have $\pi(z)=z$, then  $$ \mathcal{S}_{n}/\mathcal{K} = \left(\mathcal{S}_{2n}-
\mathcal{K}\right) \bigcup \pi(\mathcal{K}).$$
By removing from the sphere $\mathcal{S}_{2n}$  the point $N: = h(\pi(\mathcal{K})) \in \mathcal{S}_{2n}$, we obtain then a homeomorphism
$$
h: \mathcal{S}_{2n}- \mathcal{K} \longrightarrow \mathcal{S}_{2n} - \left\{ N \right\}.
$$
But we know that by stereographic projection with respect to the pole $N$, $\mathcal{S}_{2n} - \left\{N\right\}$ is already homeomorphic to
$\mathbb{C}^{n}$ . Then
$\mathcal{S}_{2n}- \mathcal{K}$ is homeomorphic to $\mathbb{C}^{n}$.
\end{proof}
\begin{corollary}\label{X est homeomorphe à C^n}$ $\\
Every $\mathcal{A}-$submanifold of $\mathcal{S}_{2n}$ is homeomorphic to $\mathbb{C}^{n}$.
\end{corollary}
\begin{proof}
Indeed, this is a particular case of the previous corollary \ref{corollaire de topologie}.
\end{proof}
\subsection{Some open sets homeomorphic to balls}
\begin{proposition}\label{proposition des boules}$ $\\
 Let $\mathcal{B}_{r}$ be an open ball of $\mathbb{C}^{n}$ and let $\mathcal{M} \subset \mathbb{C}^{n}$ be an open set homeomorphic to a ball of
 $\mathbb{C}^{n}$ such that:
 \begin{enumerate}
 \item [(a)] The boundary $\partial \mathcal{M} $ is analytic.
  \item [(b)]$ \mathcal{B}_{r} \bigcap  \mathcal{M}\neq \emptyset.$
  \item [(c)] $ \mathcal{B}_{r} \bigcap \partial \mathcal{M} $ is simply connected.
 \end{enumerate}
 Then  $ \mathcal{B}_{r} - \overline{\mathcal{M}}$ and $\mathcal{B}_{r} \bigcap \mathcal{M}$ are homeomorphic to $\mathcal{B}_{r}$.
 \end{proposition}
 \begin{proof}
Since the boundary $\partial \mathcal{M}$ is assumed to be analytic, then $ \mathcal{B}_{r} \bigcap \partial \mathcal{M}$
 admits at more a finite number $$\mathcal{K}_{1},\mathcal{K}_{2},...,\mathcal{K}_{m}$$ of $(2n-2)-$dimensional singular surfaces. By applying
 to $\mathcal{B}_{r}- \overline{\mathcal{M}}$ if necessarily a finite number $f_{1}, f_{2}, ..., f_{s}$ of "flattering" homeomorphisms as defined
 in lemma \ref{tailleur}, we obtain then a subset
$$ \mathcal{B}_{r} -  \overline{\mathcal{M}_{0}} = \left(f_{1}\circ f_{2}\circ ...\circ f_{s}\right) (\mathcal{B}_{r} -  \overline{\mathcal{M}})$$
where $\mathcal{M}_{0}$ is an open set satisfying conditions of lemma \ref{lemme d'homothetie}. Hence by lemma \ref{lemme d'homothetie},
$\mathcal{B}_{r} - \overline{ \mathcal{M}}$ is homeomorphic to $\mathcal{B}_{r}$. We prove in the same way that $\mathcal{B}_{r} \bigcap \mathcal{M}$
 is homeomorphic to $\mathcal{B}_{r}$.
The proof of proposition \ref{proposition des boules}
 is then complete.
 \end{proof}
\section{\textbf{The $\overline{\partial}-$cohomology of $\mathcal{A}-$submanifolds of $\mathcal{S}_{2n}$}}
This section is devoted to prove that the Dolbeault cohomology group of bidegree $(0,1)$ of any $\mathcal{A}-$submanifold of $\mathcal{S}_{2n}$ is
trivial.
\subsection{The hypothetic complex manifold $\left( \mathbb{C}^{n},\mathfrak{A}\right)$}
\subsubsection{\textbf{Notations and definition}}
\begin{definition}\label{New C^n complex}$ $ \\
With the notations of \ref{structure complexe de X}, let $X$ be an $\mathcal{A}-$submanifold of $\mathcal{S}_{2n}$, and let
$$\mathcal{A}\big/ X : = \bigg\{ \left( \widetilde{U_{\jmath}}, \widetilde{\varphi_{\jmath}}\right), \phantom{rrr} \jmath \in \mathcal{J}  \bigg\}.$$
be the hypothetic complex structure on $X$. By corollary \ref{corollaire de topologie}, there exists a homeomorphism
$$
h:X \longrightarrow  \mathbb{C}^{n}.
$$
The numerical space $\mathbb{C}^{n}$ is then endowed with a new(\phantom{.}\footnote{\phantom{.}The hypothetic complex structure $\mathfrak{A}$ on
$\mathbb{C}^{n}$ supposed here is of course different from the standard complex structure defined usually on $\mathbb{C}^{n}$. }) atlas $\mathfrak{A}
= h(\mathcal{A}\big/ X)$
$$\mathfrak{A}:= \bigg\{\left( V_{\jmath}, \psi_{\jmath}\right), \phantom{rrr} \jmath\in \mathcal{J} \bigg\}$$
obtained by carrying the atlas $\mathcal{A}\big/ X$ to $\mathbb{C}^{n}$. The complex charts $\left( V_{\jmath}, \psi_{\jmath}\right)$ of $\mathfrak{A}$
are then defined for all $\jmath\in \mathcal{J}$, by
\begin{equation*}
\left\{
    \begin{split}
      V_{\jmath}& = h \left(\widetilde{U}_{\jmath}\right)\\
\psi_{\jmath}& = \widetilde{\varphi}_{\jmath}\circ h^{-1}.
    \end{split}
    \right.
\end{equation*}
Taking into account the properties of $\mathcal{A}\big/ X$ given in subsubsection \ref{structure complexe de X}, the atlas $\mathfrak{A}$ inherits
then from $\mathcal{A}\big/ X$ the following properties:
\begin{enumerate}
\item  For all $\jmath\in\mathcal{J}-\left\{d\right\}$, the domain of chart $V_{\jmath}$ is bounded in $\mathbb{C}^{n}$\\
and it follows from condition (\ref{condition 1 des cartes}), that
\begin{equation}\label{condition 1 de la carte psi}
   \forall\phantom{r} \jmath \neq d, \phantom{rrr} \psi_{\jmath}\left(V_{\jmath}\right) \phantom{rrr} \textrm{is a ball of} \phantom{r} \mathbb{C}^{n}.
\end{equation}
\item It follows also from condition (\ref{condition 2 des cartes}), that 
\begin{equation}\label{condition 2 de la carte psi}
   \forall\phantom{r} \imath , \jmath \neq d, \phantom{rrr} V_{\imath} \bigcap \partial V_{\jmath} \phantom{rrr}\textrm{ is simply connected.}
\end{equation}
\item  For $\jmath = d$, the mapping
$$
 \psi_{d}: V_{d} \longrightarrow \mathcal{N}(a,r_{1},r_{2})$$
 is a homeomorphism.\\
The chart $\left(V_{d}, \psi_{d} \right)$ is said to be the distinguished chart of $Y= \left(\mathbb{C}^{n}, \mathfrak{A}\right)$.
 Observe that since for all $\jmath\in \mathcal{J}-\left\{ d\right\}$, the domains $V_{\jmath}$ are bounded, then the domain $V_{d}$ is the only one
  which is non bounded in $Y= \left(\mathbb{C}^{n}, \mathfrak{A}\right)$ and its boundary $\partial V_{d}$ satisfies
  $$ \partial V_{d} \approx \mathcal{S}_{2n-1}.$$
\item  The subset $\mathfrak{W}:= h(\mathcal{W}) = \mathbb{C}^{n} - \overline{V_{d}}$ is a relatively compact open set covered by all domains $V_{\jmath}$
 with $\jmath \neq d$
 $$
 \mathfrak{W} \subset \bigcup_{\jmath\neq d}V_{\jmath}
 $$
 and $\mathfrak{W}$ is homeomorphic to a ball of
 $\mathbb{C}^{n}$, and its boundary satisfies
    $$
    \partial \mathfrak{W}  = \partial V_{d} \approx \mathcal{S}_{2n-1}.
    $$
\end{enumerate}
 \begin{remark} $ $ \\In order to avoid any confusion with the usual complex structure of the numerical space $\mathbb{C}^{n}$, we note by
 $Y = \left(\mathbb{C}^{n}, \mathfrak{A}\right)$ the complex analytic structure of $\mathbb{C}^{n}$ defined by the hypothetic complex atlas 
 $\mathfrak{A}$
 given above.
\end{remark}
\end{definition}
\subsubsection{\textbf{The ordered atlas $(\mathfrak{A}, \preceq)$}}$ $\\
 To prove that the hypothetic $\mathcal{A}-$submanifold $X$ of $\mathcal{S}_{2n}$ is strictly pseudoconvex, we are lead by the biholomorphic mapping 
 $h: X \longrightarrow Y$ to prove that the 
 complex manifold $Y = \left( \mathbb{C}^{n}, \mathfrak{A} \right)$ is strictly pseudoconvex, which will be done by constructing an
  exhaustive smooth strictly \textbf{psh} function
 $\Psi : Y \longrightarrow \left[0,+\infty\right[.$ But to do this, we need to introduce in the hypothetic complex atlas $\mathfrak{A}$ an
  appropriate order:
 This order helps much first to construct an exhaustive almost everywhere $\mathcal{C}^{\infty}$ strictly \textbf{psh} function $f_{_{\mathfrak{A}}}:
 Y \longrightarrow
 \left[0 , +\infty\right[$, then by applying an appropriate regularizing process to $f_{_{\mathfrak{A}}}$, we construct the required strictly
 \textbf{psh}
  function $\Psi$.\\
 Let then the hypothetic complex atlas $$
 \mathfrak{A} = \bigg\{ \left(V_{\jmath},\psi_{\jmath}\right), \phantom{rrr} \jmath\in \mathcal{J}  \bigg\}
 $$
 with
 $$card (\mathcal{J}) = N+1$$ and let the relatively compact open set $$\mathfrak{W}= \mathbb{C}^{n}-\overline{V_{d}}$$ as given in definition
 \ref{New C^n complex},
where $V_{d}$ is the domain of the distinguished chart $\left( V_{d}, \psi_{d} \right)$.
 \begin{proposition}\label{proposition d'ordre}
With these notations, there exists
 a bijective mapping
 $$
 \sigma:\left\{1,2,...,N +1\right\} \longrightarrow \mathcal{J}
 $$
 such that, for all $1  \leq  k  \leq  N +1$, the open set $$\mathfrak{W}\bigcap \left(\displaystyle\bigcup_{j=1}^{k} V_{\sigma (j)}\right)$$ is
  homeomorphic to a ball of $\mathbb{C}^{n}$.
 \end{proposition}
\begin{proof}$ $\\
\textbf{\underline{$1^{st}-$ step.}} $ $ Let's organize the domains of charts of the subatlas
$$\mathfrak{A}^{\ast} : = \mathfrak{A} - \left\{\left( V_{d},\psi_{d}\right)\right\}$$ in a "spherical" arrangement. \\Indeed, by observing that
$\mathfrak{W} = \mathbb{C}^{n} - \overline{V_{d}}$ is covered by the domains of the subatlas $\mathfrak{A}^{\ast}$, and that
$$
\partial \phantom{.}\mathfrak{W} \approx \mathcal{S}_{2n-1}
$$
we define $\mathfrak{A}_{1}$ to be the set of all domains $V_{\jmath} \in\mathfrak{A}^{\ast}$ which intersect $\partial\phantom{.} \mathfrak{W}$, that is
$$
\mathfrak{A}_{1}:= \bigg\{ V_{\jmath}\in\mathfrak{A}^{\ast}, \phantom{rrrr}V_{\jmath}\cap \partial\phantom{.}\mathfrak{W}\neq \emptyset\bigg\}
$$
and we define $Y_{1}$ to be the open subset of $Y = \left( \mathbb{C}^{n}, \mathfrak{A}\right)$ covered by $\mathfrak{A}_{1}$, that is
$$
Y_{1}: = \bigcup_{V_{\jmath} \in \mathfrak{A}_{1}} V_{\jmath}.
$$
Since by property (\ref{condition 1 de la carte psi}), all domains $V_{\jmath} \in \mathfrak{A}_{1}$ are homeomorphic to balls of $\mathbb{C}^{n}$ 
and satisfy condition
(\ref{condition 2 de la carte psi}), it
follows then by proposition
\ref{proposition des boules}, that
$$
\partial Y_{1} = \Gamma_{1} \bigcup \Gamma_{2}
$$
where
\begin{equation*}\left\{
    \begin{split}
      \Gamma_{1} & \approx \partial\phantom{.} \mathfrak{W}\approx \mathcal{S}_{2n-1}\\
      & \phantom{rrr}\textrm{and}\\
      \Gamma_{2} & \approx \mathcal{S}_{2n-1} \phantom{rrr} \textrm{ or eventually} \phantom{rrr} \Gamma_{2}= \emptyset.
      \end{split}
      \right.
\end{equation*}
If $\Gamma_{2}\neq \phi$, that is, if $\Gamma_{2} \approx \mathcal{S}_{2n-1}$, we continue the construction, and we define
$$
\mathfrak{A}_{2}:= \bigg\{ V_{\jmath}\in\mathfrak{A}, \phantom{rrrr} V_{\jmath}\cap \Gamma_{2} \neq \emptyset\bigg\}
$$
and
$$
Y_{2}: = \bigcup_{V_{\jmath} \in \mathfrak{A}_{2}} V_{\jmath}.
$$
We observe as before that
$$
\partial Y_{2} = \Gamma_{2}\bigcup \Gamma_{3}
$$
where
\begin{equation*}\left\{
    \begin{split}
      \Gamma_{2} & \approx \mathcal{S}_{2n-1}\\
      & \phantom{rrr}\textrm{and}\\
      \Gamma_{3} & \approx \mathcal{S}_{2n-1} \phantom{rrr} \textrm{ or eventually} \phantom{rrr} \Gamma_{3}= \emptyset.
      \end{split}
      \right.
\end{equation*}
By induction, we obtain a finite family of "spherical" coverings
$\big\{\mathfrak{A}_{r}\big\}_{ 1 \leq r \leq m} \subset \mathfrak{A}^{\ast}$ defined by
 $$
\mathfrak{A}_{r} = \bigg\{ V_{\jmath} \in \mathfrak{A}^{\ast}, \phantom{rr} V \cap \Gamma_{r} \neq \emptyset \bigg\}
$$
and a finite family of open sets $\big\{Y_{r}\big\}_{ 1 \leq r \leq m}$ of $\left( \mathbb{C}^{n}, \mathfrak{A} \right)$, where $Y_{r}$ is covered
 by $\mathfrak{A}_{r}$, that is
$$
 Y_{r}: = \bigcup_{V_{\jmath} \in \mathfrak{A}_{r}} V_{\jmath}
$$
and where the boundary of $ Y_{r}$ is given according to proposition \ref{proposition des boules}, by
\begin{equation*}\left\{
    \begin{split}
      \partial Y_{r} & = \Gamma_{r} \bigcup \Gamma_{r+1} \phantom{rrr} \textrm{if} \phantom{r} 1 \leq r \leq m -1\\
      \partial Y_{m} & = \Gamma_{m} \approx \mathcal{S}_{2n-1}
    \end{split}\right.
\end{equation*}
and for all $ 1 \leq r \leq m$ $$ \Gamma_{r} \approx \mathcal{S}_{2n-1}.$$ Hence the subatlas $\mathfrak{A}^{\ast}$ can be written as follows
$$
\mathfrak{A}^{\ast} = \bigcup_{r = 1}^{m}\mathfrak{A}_{r}.
$$
\textbf{\underline{$2^{nd}-$ step.}} $ $ Now write
$$
\mathfrak{W} = \bigcup_{\jmath \in \mathcal{J}} \left(\mathfrak{W}\bigcap V_{\jmath}\right),
$$
Since we know that $\mathfrak{W}$ is homeomorphic to an open ball of $\mathbb{C}^{n}$, we can then by proposition \ref{proposition des boules},
 define a bijective mapping $$\sigma:\left\{1,2,...,N +1\right\} \longrightarrow \mathcal{J}$$
satisfying condition of proposition \ref{proposition d'ordre}.\\ Indeed, let us proceed as follows: First, set
$$
\sigma(N+1) : = d \in \mathcal{J}
$$
that is
$$V_{\sigma (N+1)}: = V_{d}$$
then define by induction the respective values
$$
\sigma(N)\phantom{r},\phantom{r}\sigma(N-1)\phantom{r},\phantom{r}\sigma(N-2)\phantom{r}, \phantom{r}\sigma(N-3)\phantom{r},\phantom{r}
 ...\phantom{r},\phantom{r} \sigma(3)\phantom{r},\phantom{r}\sigma(2)\phantom{r},\phantom{r}\sigma(1)
$$
by removing from $\mathfrak{W}$ the sets  $\mathfrak{W}\bigcap \overline{V_{\jmath}}$ one by one, while respecting  the following rule:
 whenever we remove a domain $\overline{V_{\sigma(j)}}$, the next domain $\overline{V_{\sigma(j-1)}}$ chosen to be removed must intersect
 the maximum number of the previous ones, that is $$ V_{\sigma(j-1)} \bigcap V_{\sigma(j)}\bigcap V_{\sigma(j+1)}\bigcap ...
 \bigcap V_{\sigma(j+k)}\neq\emptyset.$$ We begin then by removing the domains of $\mathfrak{A}_{m}$ one by one until exhausting
 $\mathfrak{A}_{m}$, then we remove the domains of $\mathfrak{A}_{m-1}$ one by one until exhausting $\mathfrak{A}_{m-1}$, and so on,
   until exhausting $\mathfrak{A}_{1}$.\\
Recall that the domains $V_{\jmath}$ with $\jmath \neq d$ satisfy condition (\ref{condition 2 de la carte psi}), and
 then we are in a position to apply proposition \ref{proposition des boules}. \\ Since $\mathfrak{W}$ is
homeomorphic to an open ball of
 $\mathbb{C}^{n}$,
 the next open set obtained by removing a chosen
 domain $\mathfrak{W}\bigcap\overline{V_{\sigma(N)}}$ from $\mathfrak{W} $, that is $\mathfrak{W} - \mathfrak{W}\bigcap\overline{V_{\sigma(N)}}$ is
 according to proposition \ref{proposition des boules} also homeomorphic to
 an open ball of $\mathbb{C}^{n}$.\\ Taking into account that the domains $V_{\jmath}$  with $\jmath \neq d$ satisfy condition
 (\ref{condition 2 de la carte psi}),
 we prove by induction that whenever
 $$
 \mathfrak{W} - \bigcup_{j =k+1}^{N} \left(\mathfrak{W}\bigcap V_{\sigma (j)}\right)
 $$
 is homeomorphic to an open ball of
 $\mathbb{C}^{n}$, then by proposition \ref{proposition des boules}, the next open set obtained by removing a chosen
 domain $\mathfrak{W}\bigcap\overline{V_{\sigma(k+1)}}$ from $\mathfrak{W} $, that is
$$ \mathfrak{W} - \bigcup_{j =k}^{N} \left(\mathfrak{W}\bigcap V_{\sigma (j)}\right)
$$
is also homeomorphic to an open ball of $\mathbb{C}^{n}$.\\
Hence, we construct a mapping
$$
\sigma:\left\{1,2,...,N +1\right\} \longrightarrow \mathcal{J}
$$
such that, for all $ 1 \leq k \leq N$, the open set
$$
\bigcup_{j =1}^{k} \left(\mathfrak{W}\bigcap V_{\sigma (j)}\right)
$$
is homeomorphic to an open ball of $\mathbb{C}^{n}$.
\end{proof}
\begin{remark}$ $
The bijective mapping $\sigma$ is not unique.
\end{remark}
\begin{corollary}
The relation $\preceq $ defined in $\mathfrak{A}$ by
\begin{equation}
    \left(V_{\sigma (j)}, \psi_{\sigma (j)} \right) \preceq \left(V_{\sigma (k)}, \psi_{\sigma (k)} \right) \Longleftrightarrow j \leq k
   \end{equation} is a total order. The atlas $\mathfrak{A}$  endowed with this relation is denoted by $(\mathfrak{A},\preceq)$.
\textbf{\underline{Notation}.}  Once the relation of order $\preceq \phantom{r}$ in $\mathcal{A}$  is acquired, we can note the charts of $\mathfrak{A}$
simply by $\left( V_{j}, \psi_{j}\right)$ instead of  $\left( V_{\sigma (j)}, \psi_{\sigma (j)}\right)$, and then the atlas $\mathfrak{A}$ can be written
$$  \mathfrak{A} = \bigg\{\left( V_{j}, \psi_{j}\right), \phantom{rrrr} 1 \leq   j   \leq N+1    \bigg\}.$$
With this notation, the distinguished chart is then $\left(V_{_{N+1}}, \psi_{_{N+1}}\right)$.
\end{corollary}
\begin{corollary}\label{Les ouverts Wk}$ $\\
Let the relatively compact set $$\mathfrak{W} = \mathbb{C}^{n}-\overline{V_{d}}.$$
With the notations above, there exists a finite increasing sequence of open
  sets $\mathfrak{W}_{k}$ all homeomorphic to balls of $\mathbb{C}^{n}$:
      $$
      \mathfrak{W}_{1}\subset \mathfrak{W}_{2}\subset \mathfrak{W}_{3}\subset... \subset\mathfrak{W}_{_{N-1}}\subset \mathfrak{W}_{_{N}}
      $$
such that $\mathfrak{W}= \mathfrak{W}_{_{N}}$.
\end{corollary}
\begin{proof}
Indeed, it suffices according to proposition \ref{proposition d'ordre}, to take for all $ 1 \leq k \leq N$
      \begin{equation}\label{Les ouverts Wk homeomorphes à des boules}
        \mathfrak{W}_{k}= \bigcup _{j=1}^{k}\left( \mathfrak{W}\bigcap V_{j}\right).
      \end{equation}
\end{proof}
\subsubsection{\textbf{Points of first and points of second kind with respect to $\left(\mathfrak{A},\preceq \right)$}}$ $\\
It emerges from the relation of order $\preceq$ on the atlas $\mathfrak{A}$ a natural distinction between points of $Y = \left( \mathbb{C}^{n},
 \mathfrak{}\right)$. This distinction is made precise in the following definition.
\begin{definition}
We say that a point $z \in Y = \left( \mathbb{C}^{n}, \mathfrak{A}\right)$ is of first kind with respect to the ordered atlas $(\mathfrak{A},
\preceq)$, if $z$ satisfies the following condition
\begin{equation}\label{point de 1 espece}
\textrm{For all} \phantom{r}  j,k \in \left\{  1,...,N+1\right\}\phantom{r} \textrm{with} \phantom{r} j\leq k, \phantom{rrr} z 
\not\in \left(\partial V_{j}\right) \bigcap V_{k}.
\end{equation}
A point $z \in Y =\left( \mathbb{C}^{n}, \mathfrak{A}\right)$ is said to be of second kind with respect to $(\mathfrak{A},\preceq)$, if
\begin{equation}\label{point de 2 espece}
\textrm{There exist} \phantom{r}  j,k \in \left\{1,...,N+1\right\}\phantom{r} \textrm{with} \phantom{r} j\leq k, \phantom{rr} \textrm{and}
 \phantom{rr} z \in \left(\partial V_{j}\right) \bigcap V_{k}.
\end{equation}
The condition for a point $z$ to be of second kind with respect to $(\mathfrak{A},\preceq)$, means that $z$ belongs 
to a domain of chart $V_{k}$, and in the same time to the boundary of a previous one.\\
\textbf{\underline{Notation}.} We note by $\mathbb{K}_{1}$ respectively $\mathbb{K}_{2}$ the sets of points of first kind respectively of
 second kind with respect to $(\mathfrak{A},\preceq)$.
\end{definition}
\begin{remark}
\begin{enumerate} \emph{
\item Not any point of the boundary of a domain of chart is of second kind with respect to $(\mathfrak{A},\preceq)$, because 
$\partial V_{_{N+1}}\not\subset \mathbb{K}_{2}$.
\item The set $\mathbb{K}_{2}$ is neglected with respect to the Lebeague measure, because
$$\mathbb{K}_{2}   \subset \bigcup_{j=1}^{N} \partial V_{j}$$
\item $\mathbb{K}_{1}$ is an open subset of $Y = \left(\mathbb{C}^{n}, \mathfrak{A}\right)$.}
\end{enumerate}
\end{remark}
\subsection{The almost everywhere psh function $f_{\mathfrak{A}}$}$ $\\\\
We show below how we can associate to the ordered atlas $\left( \mathfrak{A}, \preceq\right)$ an exhaustive function
$f_{_{\mathfrak{A}}}: Y = \left( \mathbb{C}^{n}, \mathfrak{A} \right)\longrightarrow \left[0, +\infty\right[$
 almost everywhere $\mathcal{C}^{\infty}$ and strictly psh.
\subsubsection{Notations and definitions}\label{La fonction f}$ $\\
With the notations of definition \ref{New C^n complex}, let the hypothetic ordered atlas
$$
\mathfrak{A} = \bigg\{ \left( V_{j}, \psi_{j}\right) , \phantom{rrrr} 1 \leq j \leq N +1\bigg\}
$$
and let $\mathfrak{W}$ be the relatively compact open subset of $Y=\left(\mathbb{C}^{n}, \mathfrak{A}\right)$ defined by
$$
\mathfrak{W}= \mathbb{C}^{n}-\overline{ V_{N+1}},
$$
recall that $\mathfrak{W}$ is covered by the domains of charts $V_{j}$, with $ 1 \leq j \leq N$
$$
\mathfrak{W} \subset \bigcup_{j=1}^{N}V_{j}.
$$
Now, let $a_{1}\in V_{1}$, such that $\psi_{1}\left( a_{1} \right)$ is the center of the ball $\psi_{1}\left(V_{1}\right)$, and chose in the open set
 $\mathfrak{W}$, $\phantom{r}N-1\phantom{r} $ points of second kind
$$
\bigg\{a_{2}, a_{3},..., a_{_{N}}\bigg\} \subset \mathfrak{W}\bigcap\mathbb{K}_{2}
$$
such that, for $2 \leq k \leq N$
$$
a_{k}\in \left(\partial V_{k-1}\right) \bigcap V_{k}
$$
and define by induction the real valued functions
$$f_{1}\phantom{.} ,\phantom{.} f_{2}\phantom{.},\phantom{.} f_{3}\phantom{.},\phantom{.}...\phantom{.},\phantom{.}f_{_{N}}\phantom{.},\phantom{.}f_{_{N+1}}$$
as follows:\\
The function $f_{1}$ is defined on $Dom\left(f_{1}\right) = \overline{V_{1}}\bigcap \overline{\mathfrak{W}}$ by
\begin{equation}\label{f_(1)}
    f_{1}(z) = \left\|\psi_{1}(z) - \psi_{1}\left( a_{1}\right)\right\|^{2}
\end{equation}
and for $\phantom{r} 2 \leq  k  \leq N$, the function $f_{k}$ is defined on $Dom \left(f_{k}\right) = \left(\overline{V_{k}} - \overline{V_{k-1}}\right)\bigcap \overline{\mathfrak{W}}$
 by
 \begin{equation}\label{f_(j)}
    f_{k}(z) = \left\|\psi_{k}(z) - \psi_{k}\left(a_{k}\right)\right\|^{2}  + b_{k-1}
 \end{equation}
 and for $k = N+1$, the function $f_{_{N+1}}$ is defined on $Dom\left(f_{_{N+1}}\right) = V_{_{N+1}} = \mathbb{C}^{n}- \overline{\mathfrak{W}}$
by
\begin{equation}\label{f_(N+1)}
    f_{_{N+1}}(z) = -log\left(\max \left\{\left(\frac{ \left\|\psi_{_{N+1}}(z)\right\|^{2}-r_{2}^{2}}{r_{1}^{2} - r_{2}^{2}}\right)^{2} ,
 \phantom{r} \left(\frac{Re\big\langle \psi_{_{N+1}}(z),w\big\rangle}{r_{1}.\|w\|}\right)^{2}\phantom{.}\right\}\right) + b_{_{N}}.
\end{equation}
where for all $ 1\leq k \leq N$, the numbers $b_{k}$ are given by induction by
\begin{equation}\label{Le sup bk}
    b_{k} =  \sup_{z\in Dom(f_{k})} f_{k}(z).
\end{equation}
By gluing the functions $f_{1}, f_{2},...,f_{_{N+1}}$ defined by (\ref{f_(1)}), (\ref{f_(j)}), (\ref{f_(N+1)}) together, we define then the following global
function  $$f_{_{\mathfrak{A}}} : \left( \mathbb{C}^{n}, \mathfrak{A} \right) \longrightarrow \left[ 0, +\infty \right[$$ by
\begin{equation}\label{La fonction presque psh f}
     f_{_{\mathfrak{A}}}(z): = \left\{\begin{split}
               & f_{1}(z) \phantom{rrrr}\textrm{if} \phantom{rrr} z\in V_{1}\bigcap\mathfrak{W} \\
              &  f_{k}(z) \phantom{rrrr} \textrm{if} \phantom{rrr} z\in \left(V_{k}-V_{k-1}\right)\bigcap\mathfrak{W},\phantom{rrrr}
               \textrm{with}\phantom{r} 2 \leq k \leq N\\
               &  f_{_{N+1}}(z) \phantom{rr}\textrm{if} \phantom{rrr} z\in \overline{V_{_{N+1}}}.
             \end{split}
             \right.
\end{equation}
\begin{remark}
Observe that the function $f_{_{\mathfrak{A}}}$ defined by (\ref{La fonction presque psh f}) is not continuous, and that the points of discontinuity
 of $f_{_{\mathfrak{A}}}$ belong to
 the set of second kind $\mathbb{K}_{2}$.
\end{remark}
\begin{lemma}\label{lemme de la fonction tilde(f_N+1)}$ $\\ For $ r_{1} > r_{2} > 0$, let
$$\overline{\mathcal{N}}\left( a, r _{1}, r_{2}\right): = \overline{\mathcal{B}}\left( a, r_{1}\right) -
\mathcal{S}_{2n-2}\left(0,r_{2}\right)$$ and
 let $ \widetilde{f}_{_{N+1}} = f_{_{N+1}}\circ \psi^{-1}_{_{N+1}}$
 be the representation in the local coordinates $\zeta = \psi_{_{N+1}}(z)$ of the function $f_{_{N+1}}$  defined in (\ref{f_(N+1)}), that is,
 $
 \widetilde{f}_{_{N+1}}: \overline{\mathcal{N}}\left( a, r _{1}, r_{2}\right) \longrightarrow \mathbb{R}
 $
 defined by
 $$
 \widetilde{f}_{_{N+1}}(\zeta) = -log\left(\max \left\{\left(\frac{ \left\|\zeta\right\|^{2}-r_{2}^{2}}{r_{1}^{2} - r_{2}^{2}}\right)^{2} ,
 \phantom{r} \left(\frac{Re\langle \zeta,w\rangle}{r_{1}.\|w\|}\right)^{2}\phantom{.}\right\}\right) + b_{_{N}}.
 $$
Then the function $\widetilde{f}_{_{N+1}}$
satisfies the following properties:
\begin{enumerate}
  \item $\widetilde{f}_{_{N+1}}$ is almost everywhere a $\mathcal{C}^{\infty}$ strictly \textbf{psh},
  \item The following equivalence holds
\begin{equation}\label{equation du lemme 2}
\widetilde{f}_{_{N+1}}(\zeta)=0 \Longleftrightarrow \zeta \in S_{2n-2}\left(0,r_{2}\right),
\end{equation}
\item For all $ c > b_{_{N}}$, there exists $ 0 <  \rho < r_{1}$ such that the sublevel set
\begin{equation}\label{L'ensemble de niveau F_{c}}
    L_{c}\left( \widetilde{f}_{_{N+1}}\right): = \bigg\{ \zeta \in \overline{\mathcal{N}}\left( a, r _{1}, r_{2}\right),
    \phantom{rrr} \widetilde{f}_{_{N+1}}(\zeta) \leq c  \bigg\}
\end{equation}
is homeomorphic to the closed annulus $
\overline{\mathcal{B}}\left( a, r_{1}  \right) - \mathcal{B}\left( a, \rho \right)$
\end{enumerate}
\end{lemma}
\begin{proof}
Consider the functions
$$
g_{1}(\zeta) = \left(\frac{ \left\|\zeta\right\|^{2}-r_{2}^{2}}{r_{1}^{2} - r_{2}^{2}}\right)^{2}
$$
and
$$
g_{2}(\zeta) =  \left(\frac{Re\langle \zeta,w\rangle}{r_{1}.\|w\|}\right)^{2}.
$$
1) The function $g_{1}$ is clearly $\mathcal{C}^{\infty}$ and is strictly psh  because it is composed of the strictly psh function
$\frac{ \left\|\zeta\right\|^{2}-r_{2}^{2}}{r_{1}^{2} - r_{2}^{2}}$ with the strictly convex function $x \longmapsto x^{2}$.\\
The function $g_{2}$ is clearly of class $\mathcal{C}^{\infty}$ and is also strictly psh because we have
$$
g_{2}(\zeta) =  \left(\frac{Re\langle \zeta,w\rangle}{r_{1}.\|w\|}\right)^{2} =
\frac{\left(\langle \zeta,w\rangle\right)^{2} + \left(\langle \overline{\zeta},\overline{w}\rangle\right)^{2} + 2
\left(\left|\langle \zeta,w\rangle\right|^{2}\right)}{4 \left(r_{1}.\|w\|\right)^{2}}
$$
and then the Levi form of $g_{2}$ coincides with the levi form of the function $$\zeta \longmapsto \frac{\left|\langle
\zeta,w\rangle\right|^{2}}{2\left(r_{1}.\|w\|\right)^{2}}$$
which is clearly strictly psh. Since the set
$$
\bigg\{ z \in \mathcal{N}(a,r_{1},r_{2}), \phantom{rr} \left|g_{1}(\zeta)\right|= \left|g_{2}(\zeta)\right| \bigg\}
$$
is neglected with respect to Lebeasgue measure, it follows  then that $\widetilde{f}_{_{N+1}}$ is almost everywhere a $\mathcal{C}^{\infty}$ strictly
 \textbf{psh} function. \\
2) It is clear that the set of solutions of the equation $\widetilde{f}_{_{N+1}}(\zeta)=0 $ is
$$
S_{2n-2}\left(0,r_{2}\right) = \bigg\{ \zeta\in \mathbb{C}^{n}, \phantom{rrr} \widetilde{f}_{_{N+1}}(\zeta)=0   \bigg\}.
$$
3) Let the sets
$$
F_{1} = \bigg\{ \zeta \in \mathcal{N}(a,r_{1},r_{2}), \phantom{rr} \left| g_{1}(\zeta)\right| < e^{-c} - b_{_{N}} \bigg\}
$$
and
$$
F_{2} = \bigg\{ \zeta \in \mathcal{N}(a,r_{1},r_{2}), \phantom{rr} \left| g_{2}(\zeta) \right| < e^{-c} - b_{_{N}} \bigg\}.
$$We check easily that $F_{1}\bigcap F_{2}$ is a bounded open convex set, and then $F_{1}\bigcap F_{2}$ is homeomorphic to a ball $\mathcal{B}(a, \rho)$
 of $\mathbb{C}^{n}$.
Since $L_{c}\left( \widetilde{f}_{_{N+1}}\right)$ is the complementary of $F_{1}\bigcap F_{2}$ with respect to the closed ball
$\overline{\mathcal{B}}(a,r_{1})$,
then $L_{c}\left( \widetilde{f}_{_{N+1}}\right)$
 is homeomorphic to a closed annulus of the form $\overline{\mathcal{B}}\left( a, r_{1}  \right) - \mathcal{B}\left( a, \rho \right)$. The proof of the
 lemma \ref{lemme de la fonction tilde(f_N+1)} is then complete.
\end{proof}
\begin{proposition}\label{Proprietes de la fonctions f}\textbf{(Properties of the function $f_{\mathfrak{_{A}}}$)}$ $\\
Let the function $f_{_{\mathfrak{A}}}: Y = \left(\mathbb{C}^{n}, \mathfrak{A} \right)\longrightarrow \left[ 0 , + \infty \right[$  defined by
(\ref{La fonction presque psh f}), and for $c > 0$, let the sublevel set
$$L_{c}\left(f_{_{\mathfrak{A}}}\right) = \bigg\{z \in Y, \phantom{rrr} f_{\mathfrak{A}}(z) \leq c   \bigg\}.$$
Then \begin{enumerate}
       \item  On the set of first kind $\mathbb{K}_{1}$, $f_{_{\mathfrak{A}}}$ is of class $\mathcal{C}^{\infty}$.
       \item On the set of first kind $\mathbb{K}_{1}$, $f_{_{\mathfrak{A}}}$ is strictly \textbf{psh.}
       \item For all $c > 0$, the set $L_{c}\left(f_{_{\mathfrak{A}}}\right)$ is homeomorphic to a closed ball of $\mathbb{C}^{n}$.
     \end{enumerate}
\end{proposition}
\begin{proof}$ $\\
(1) By recalling the definition of the open set of the points of first kind $\mathbb{K}_{1}$ and by examining the expressions
 of the different functions $f_{k}$,
 we observe clearly  that the function $f_{_{\mathfrak{A}}}$ is of class $\mathcal{C}^{\infty}$ in $\mathbb{K}_{1}$.\\
(2) Since the property of positivity of the Levi form of a function, is independent of the choice of local coordinates, it suffices then
 for a point $z \in V_{k}$ to check this positivity only in the chart $\left( V_{k}, \psi_{k} \right)$. Let then $z \in \mathbb{K}_{1}$, and for
 all $1 \leq k \leq N+1$,
 let $\widetilde{f}_{k} = f \circ \psi_{k}^{-1}$ be the expression of $f_{k}$ in the local coordinates $\zeta = \psi_{k}(z)$.
 By a careful examination of the different expressions of the function $f_{k}$, we obtain:
\begin{enumerate}
  \item [(a)]If $z \in V_{1}\bigcap \mathfrak{W}$, then
  $$
  \widetilde{f}_{1}(\zeta) = \left\|\zeta-a_{1}\right\|^{2}
  $$
 which means that the function $f_{1}$ is strictly \textbf{psh} on $V_{1}\bigcap \mathfrak{W}$.
  \item [(b)]For $ 2 \leq k \leq N$, and $z \in \left(V_{k} - \overline{V_{k-1}}\phantom{.}\right) \bigcap \mathfrak{W}$,  we have
  $$ \widetilde{f}_{k} (\zeta) =  \left| \zeta- \psi_{k}\left( a_{k}\right)\right|^{2} + b_{k-1}.$$
  The function $f_{k}$ in also strictly \textbf{psh} on $\left(V_{k} - \overline{V_{k-1}}\phantom{.}\right) \bigcap \mathfrak{W}$.
  \item [(c)] If $z \in \overline{V_{_{N+1}}}$, we know by lemma \ref{lemme de la fonction tilde(f_N+1)}, that the function
      $$\widetilde{f}_{_{N+1}}(\zeta) = -log\left(\max \left\{\left(\frac{ \left\|\zeta\right\|^{2}-r_{2}^{2}}{r_{1}^{2} - r_{2}^{2}}\right)^{2} ,
 \phantom{r} \left(\frac{Re\langle \zeta,w\rangle}{r_{1}.\|w\|}\right)^{2}\phantom{.}\right\}\right) + b_{_{N}}.$$
  is almost everywhere strictly \textbf{psh} on $\psi_{_{N+1}}\left(V_{_{N+1}}\right)$.
\end{enumerate}
We conclude then by this discussion that on the set of points of first kind $\mathbb{K}_{1}$, the function $f_{\mathfrak{A}}$ is strictly
\textbf{psh}, and therefore its Levi form  given in the local coordinates $\zeta = \psi_{k}(z)$ by
 \begin{equation}\label{La forme de Levi de f}
    \mathscr{L}_{\zeta}(\widetilde{f}_{k})\left[t,t\right] = \sum_{j,l}\frac{\partial^{2}\widetilde{f}_{k}(\zeta)}{\partial\zeta _{j}\partial \overline{\zeta}_{l}}t_{j}\overline{t}_{l}\phantom{rrrrrrr} \textrm{where} \phantom{r}t = \left(t_{1}, ..., t_{n}\right)\in \mathbb{C}^{n}
\end{equation}
is positive defined on $\mathbb{K}_{1}$ as quadratic form of the variable $t \in \mathbb{C}^{n}$.\\
(3) Let $c > 0$, and let the sublevel set
$$
L_{c}(f_{_{\mathfrak{A}}}): = \bigg\{ z \in Y, \phantom{rrrr} f_{_{\mathfrak{A}}}(z) \leq c  \bigg\}.
$$
\textbf{\underline{$1^{st}$ case.}} $\phantom{rr}c = b_{k}.$ \\
By the definition of the number $b_{k}$ given by equation (\ref{Le sup bk}), we have
$$
L_{b_{k}}(f_{_{\mathfrak{A}}}) = \overline{\mathfrak{W}_{k}}.
$$
and then $L_{b_{k}}(f_{_{\mathfrak{A}}})$ is homeomorphic to a closed ball of $\mathbb{C}^{n}$. \\
\textbf{\underline{$2^{nd}$ case.} $\phantom{rr} b_{k} < c < b_{k+1}.$}\\
Let the closed ball $\overline{\mathcal{B}}\left( a_{k+1}, c- b_{k}\right),$  since $\psi_{k+1}\left( V_{k+1}\right)$ is a ball,
then the set
$$ D_{c,k}\left( f_{_{\mathfrak{A}}}\right): =
\psi_{k+1}^{-1}\left( \overline{\mathcal{B}}\left( a_{k+1}, c- b_{k}\right)\bigcap  \overline{\psi_{k+1}\left( V_{k+1}\right)}\right)$$
is homeomorphic to a closed ball of $\mathbb{C}^{n}$. For $b_{k} < c < b_{k+1}$, we can write $L_{c}\left(f_{_{\mathfrak{A}}}\right)$ as follows:
$$
L_{c}\left(f_{_{\mathfrak{A}}}\right) = \overline{{\mathfrak{W}}_{k+1}} -
\bigg(\bigg(\overline{V_{k+1}}-\mathfrak{W}_{k}\bigg) - \bigg( D_{c,k}\left( f_{_{\mathfrak{A}}}\right) - \mathfrak{W}_{k}\bigg)\bigg).
$$
By proposition \ref{proposition des boules}, we deduce that $L_{c}\left(f_{_{\mathfrak{A}}}\right)$ is
homeomorphic to a closed ball of $\mathbb{C}^{n}$.\\
\textbf{\underline{$3^{rd}$ case.} $\phantom{rr} b_{_{N}} < c.$}\\
In this case we have with the notations of lemma \ref{lemme de la fonction tilde(f_N+1)}
\begin{equation}\label{reunion}
    L_{c}\left(f_{_{\mathfrak{A}}}\right) =  L_{b_{N}}\left(f_{_{\mathfrak{A}}}\right)
 \bigcup \psi_{_{N+1}}^{-1} \left(L_{c}\left( \widetilde{f}_{_{N+1}}\right)\right).
\end{equation}
and
\begin{equation}\label{intersection}
    \partial L_{c}\left(f_{_{\mathfrak{A}}}\right) \bigcap \partial \left(\psi_{_{N+1}}^{-1} \left(L_{c}\left( \widetilde{f}_{_{N+1}}\right)\right)\right) =
 \partial V_{_{N+1}} \approx \mathcal{S}_{2n-1}\left(a, r_{1}\right).
\end{equation}
Since $L_{b_{N}}\left(f_{_{\mathfrak{A}}}\right)$ is homeomorphic to the ball
and from lemma \ref{lemme de la fonction tilde(f_N+1)}, $\psi_{_{N+1}}^{-1} \left(L_{c}\left( \widetilde{f}_{_{N+1}}\right)\right)$
is homeomorphic to an annulus,
we deduce from (\ref{reunion}) and (\ref{intersection}) that $L_{c}\left(f_{_{\mathfrak{A}}}\right)$ is homeomorphic to a closed ball of $\mathbb{C}^{n}$.
The proof of proposition \ref{Proprietes de la fonctions f} is then complete.
\end{proof}
\subsection{A vanishing theorem}
\subsubsection{\textbf{A regularization process}}\label{procede de regularisation}$ $\\
In this subsection, we show how the properties of the atlas $\mathfrak{A}$ enable us to introduce on the hypothetic complex manifold
$Y = \left( \mathbb{C}^{n}, \mathfrak{A}\right)$ an appropriate regularization process.\\
Indeed, let $1 \leq j \leq  N+1$ and let $\varepsilon >0$ small enough, and consider the open set
$$
V_{j}^{\varepsilon}: = \bigg\{ z\in V_{j}, \phantom{rr} d\bigg(\psi_{j}(z), \partial\big(\psi_{j}\left( V_{j}\big)\right)  \bigg) > \varepsilon   \bigg\}.
$$
Since the domain $V_{j}$ is assumed to be a ball of $\mathbb{C}^{n}$, then $V_{j}^{\varepsilon}$ is also a ball of $\mathbb{C}^{n}$. \\
 Fix $\varepsilon > 0$ so small that the union of open sets $V_{j}^{\varepsilon}$ cover $\mathbb{C}^{n}$. \\
We want now to associate to the atlas $\mathfrak{A}$ a matrix regularizing  operator. \\
For this, let $\mathcal{B}_{1}$ be the unit ball of $\mathbb{C}^{n}$, and let $\chi \in \mathcal{C}^{\infty}_{0}\big( \mathcal{B}_{1}\big)$
be a cutoff function with $\chi(\zeta) \geq 0$, and $\int\limits_{\mathcal{B}_{1}}\chi(\zeta)d\zeta=1$. We note $\chi_{\varepsilon}(z) =
\chi\left( \frac{z}{\varepsilon} \right)$.\\
Since the open sets $\psi_{j}\left( V_{j} \right)$ are balls of $\mathbb{C}^{n}$, then by translating them if necessarily and independently of one another,
we can always assume that
for all $ j,k \in \left\{ 1,...,N+1\right\}$, and for all $\left( \zeta_{j}, \zeta_{k}\right) \in \psi_{j}\left( V_{j}\right)
\times \psi_{k}\left( V_{k} \right)$, the following condition is fulfilled
\begin{equation}\label{diagonalisation de l'operateur regularisant}
    \zeta_{j} - \zeta_{k} \in \mathcal{B}_{\varepsilon} \Longleftrightarrow j = k\phantom{rrrrrrr} (\mathcal{B}_{\varepsilon}\phantom{r}\textrm{is the
     ball of radius}\phantom{r} \varepsilon).
\end{equation}
That is, in other words, one can translate in $\mathbb{C}^{n}$ the domains of coordinates $\psi_{j}\left(V_{j}\right)$ so far away, that for $j \neq k$,
 the points of  $\psi_{j}\left(V_{j}\right)$ become fairly distant from the points of $\psi_{k}\left(V_{k}\right)$ that
$\zeta_{j} - \zeta_{k} \not\in \mathcal{B}_{\varepsilon}$.\\
 Let $ \phantom{r}\mathbb{L}^{1}_{loc}(Y)$ be the space of locally integrable complex valued functions with respect to Lebeasgue measure $$dz =
 \frac{1}{(2i)^{n}}\bigwedge _{j=1}^{n}\left( d\overline{z}_{j}\wedge dz_{j}\right),$$
 and let the mapping $\widehat{\psi}_{j}: Y \longrightarrow \mathbb{C}^{n}$ given by:
$$
\widehat{\psi}_{j}(z):= \left\{
\begin{split}\label{psi tilde}
  \psi_{j} (z)& \phantom{rrr} \textrm{if}\phantom{r} z \in V_{j}\\
   0 \phantom{rr}& \phantom{rrr} \textrm{if}\phantom{r} z \not\in V_{j}.
\end{split}
\right.
$$
With these notations, we define then a linear $(N+1)\times(N+1)-$matrix operator
$$
\mathcal{R}_{\mathfrak{A}}: \mathbb{L}^{1}_{loc}\left( Y\right)\longrightarrow \big(\mathbb{L}^{1}_{loc}
\left( Y \right)\big)^{\left(N+1\right)^{2}}
$$
$$
\phantom{rrrrrrrr}f \longmapsto \mathcal{R}_{\mathfrak{A}}\left[ f\right] = \bigg(\mathcal{R}_{j,k}
\left[ f\right]\bigg)_{1 \leq j,k \leq N+1}
$$
by
\begin{equation}\label{Operateur matriciel regularisant}
    \mathcal{R}_{j,k}\left[f\right](z) = \varepsilon^{2n}\int\limits_{\xi \in V_{j}^{2\varepsilon}}f(\xi).\chi_{\varepsilon}
    \left(\widehat{\psi}_{k}(z)-\widehat{\psi}_{j} (\xi) \right) d \widehat{\psi}_{j} (\xi).
\end{equation}
\begin{proposition}\label{propriétés de l'operateur de regularisation}$ $\\
For all $f \in \mathbb{L}^{1}_{loc}(Y)$, we have the following properties of $\mathcal{R}_{\mathfrak{A}}\left[ f \right]$:
\begin{enumerate}
 \item The linear operator $\mathcal{R}_{\mathfrak{A}}$ is diagonal, that is
$$ \mathcal{R}_{j,k}\left[f\right] = 0 \Longleftrightarrow j\neq k $$

 \item The linear operator $\mathcal{R}_{\mathfrak{A}}$ is a regularizing operator, that is
 $$
 \mathcal{R}_{j,k}\left[f\right] \in \mathcal{C}^{\infty}(Y).
 $$
\item For all $z \not\in V_{k}^{\varepsilon}$, we have
$$\mathcal{R}_{k,k}\left[f\right](z)=0$$
which means that $supp \left(\mathcal{R}_{k,k}\left[f\right]\right)\subseteq \overline{V_{k}^{\varepsilon}}.$
\item By the following change of variables $\eta = \widehat{\psi}_{k}(z)-\widehat{\psi}_{k} (\xi)$, we obtain
$$
\mathcal{R}_{k,k}\left[f\right](z) = \varepsilon^{2n}\int\limits_{\eta \in \mathcal{B}_{\varepsilon}}
\left(f\circ \psi_{k}^{-1}\right)\left(\widehat{\psi}_{k}(z)-\eta\right).\chi_{\varepsilon}\left( \eta \right) d\eta.
$$
    \item Let $\widetilde{f}_{k}= f \circ \psi_{k}^{-1}$ denote the expression of the function $f$ in the local coordinates $\zeta = \psi_{k}(z)$,
     and let $\widetilde{\mathcal{R}_{k,k}\left[f\right]}= \mathcal{R}_{k,k}\left[f\right]\circ \psi_{k}^{-1}$ denote the
      expression of the function $\mathcal{R}_{k,k}\left[f\right]$ in the same coordinates $\zeta = \psi_{k}(z)$,
    \begin{displaymath}
\xymatrix {V_{k} \ar[d]_{\psi_{k}} \ar[r]^{f\phantom{r}}& \mathbb{C}  \\
\psi_{k}\left( V_{k}\right) \ar[ur]_{\widetilde{f}_{k}} & }
\phantom{rrrrrrrr}
\xymatrix {V_{k} \ar[d]_{\psi_{k}} \ar[r]^{\mathcal{R}_{k,k}\left[f\right]\phantom{rr}}& \mathbb{C} \\
\psi_{j}\left( V_{j}\right) \ar[ur]_{\widetilde{\mathcal{R}_{k,k}\left[f\right]}} & }
\end{displaymath}
Then
\begin{equation}\label{l'expression de Fj dans la carte Vj}
    \widetilde{\mathcal{R}_{k,k}\left[f\right]}(\zeta) = \varepsilon^{2n}\int\limits_{\eta \in \mathcal{B}_{\varepsilon}}
    \widetilde{f}_{k}(\zeta-\eta).\chi\left( \frac{\eta}{\varepsilon}\right)d\eta.
\end{equation}
\end{enumerate}
Using (\ref{l'expression de Fj dans la carte Vj}), we can write $\mathcal{R}_{k,k}$ in terms of convolution operation as follows
\begin{equation}\label{formule de convolution}
    \mathcal{R}_{k,k}\left[f\right] = \bigg(\left(f \circ\psi_{k}^{-1}\right)\ast \chi_{\varepsilon}\bigg)\circ \psi_{k}.
\end{equation}
\end{proposition}
\begin{proof}$ $\\
Property (1) follows from the fact that $\chi_{\varepsilon}$ is of class $\mathcal{C}^{\infty}$. \\
Property (2) is a consequence of condition (\ref{diagonalisation de l'operateur regularisant}) an the fact that
$supp \chi_{\varepsilon} \subset \mathcal{B}_{\varepsilon}.$\\
Property (3) holds because to compute $\mathcal{R}_{k,k}\left[f\right]$ we integrate on $\xi\in V_{k}^{2\varepsilon}$ and if $z \not\in V_{k}^{\varepsilon}$, then
 $\chi_{\varepsilon}\left(\widehat{\psi}_{k}(z) - \widehat{\psi}_{k}(\xi)\right) = 0$ and therefore $\mathcal{R}_{k,k}\left[f\right] = 0$.\\
 Properties (4) and (5) follow from the definition of $\mathcal{R}_{k,k}$ given in (\ref{Operateur matriciel regularisant}).
\end{proof}
\begin{corollary}\label{derivations du regularisé}
Fix a point $z_{0} \in \partial V_{k}^{\varepsilon}$, and let $z \in V_{k}^{\varepsilon}$. Then for all multi-indexes
$\alpha , \beta \in \mathbb{N}^{n}$ we obtain
$$ \lim_{z \longrightarrow z_{0}} \partial_{z}^{\alpha}\partial_{\overline{z}}^{\beta}\mathcal{R}_{k,k}\left[f\right](z) =  0.$$
\end{corollary}
\begin{proof}
This is a consequence of properties (1) and (3) of the operator $\mathcal{R}_{\mathfrak{A}}$.
\end{proof}
We can now state the main result of this section.
\begin{theorem}\label{La variete Y est pseudoconvexe}$ $\\
The hypothetic complex analytic manifold
$Y =(\mathbb{C}^{n},\mathfrak{A})$ is strictly pseudoconvex.
\end{theorem}
Before proving theorem \ref{La variete Y est pseudoconvexe}, let us deduce the following corollary.
\begin{corollary}\label{theoreme d'annulation}\textbf{(Vanishing theorem.)}$ $\\
Let $X$ be an $\mathcal{A}-$submanifold of $\mathcal{S}_{2n}$, and let $\mathcal{H}^{0,1}\left(X,\mathbb{C}^{n}\right)$ be its Dolbeault group of
cohomology of type $(0,1)$ with
coefficients in $\mathbb{C}^{n}$. Then $$\mathcal{H}^{0,1}\left(X,\mathbb{C}^{n}\right) = \left\{ 0\right\}.$$
\end{corollary}
\begin{proof}(of corollary)\\
The hypothetic complex manifold $Y = \left(\mathbb{C}^{n}, \mathfrak{A}\right)$ is by construction biholomorphic to the $\mathcal{A}-$submanifold $X$. The corollary follows from Cartan's theorem B; see \cite{12}.
\end{proof}
\begin{proof}(of theorem \ref{La variete Y est pseudoconvexe})$ $\\
Let the function $f_{_{\mathfrak{A}}}$ defined by (\ref{La fonction presque psh f}). Since $f_{_{\mathfrak{A}}}$ is almost everywhere of
class $\mathcal{C}^{\infty}$
and then locally integrable,
 we can then apply to $f_{_{\mathfrak{A}}}$ the regularizing operator $\mathcal{R}_{\mathfrak{A}}$ defined by equation
 (\ref{Operateur matriciel regularisant}).
For $\varepsilon > 0$ small enough, let then the function
$$\Psi : Y = \left( \mathbb{C}^{n},\mathfrak{A} \right) \longrightarrow \left[ 0 , + \infty \right[$$ defined by
\begin{equation}\label{La fonction phi}
    \Psi(z) = trace \left( \mathcal{R}_{\mathfrak{A}}\left[f\right]\right)(z).
\end{equation}
Let $z\in V_{k}^{\varepsilon}$ and consider the following subset of $\left\{ 1 ,2,...,N \right\}$
 $$ Ind(z): = \bigg\{ k, \phantom{rrr} \textrm{with}\phantom{r} 1 \leq k \leq N, \phantom{rrr} z \in V_{k}^{\varepsilon}  \bigg\}.$$
We can then write (\ref{La fonction phi}) as follows
\begin{equation*}
\begin{split}
    \Psi(z) & = trace \left( \mathcal{R}_{\mathfrak{A}}\left[f\right]\right)(z) \\
    & =\sum_{k=1}^{N+1} \mathcal{R}_{k,k}\left[f\right](z) \\
 & = \sum_{k=1}^{N+1}\varepsilon^{2n}\int\limits_{\xi \in V_{k}^{2\varepsilon}}f(\xi).\chi_{\varepsilon}\left(\widehat{\psi}_{k}(z)-
 \widehat{\psi}_{k} (\xi) \right) d \widehat{\psi}_{k} (\xi)\\
 & =  \varepsilon^{2n}\sum_{k \in Ind(z)}\phantom{r}\int\limits_{\xi \in V_{k}^{2\varepsilon}}f(\xi).
      \chi_{\varepsilon}\left(\widehat{\psi}_{k}(z)-\widehat{\psi}_{k} (\xi) \right)
       d \widehat{\psi}_{k}.
 \end{split}
\end{equation*}
We claim that
\begin{enumerate}
  \item $\Psi$ is a $\mathcal{C}^{\infty}$ function on $Y = \left( \mathbb{C}^{n} , \mathfrak{A}\right)$.
  \item $\Psi$ is strictly \textbf{psh} on $Y = \left( \mathbb{C}^{n} , \mathfrak{A}\right)$.
  \item $\Psi$ is an exhaustion of $Y = \left( \mathbb{C}^{n},\mathfrak{A} \right)$, that is, for all $c > 0$,  the sublevel set
  $$
  L_{c}(\Psi):= \bigg\{z\in Y, \phantom{rrrr} \Psi(z) < c   \bigg\}
  $$
  is relatively compact.
\end{enumerate}
Indeed. \\
1) Fix $z_{0}\in Y = \left(\mathbb{C}^{n}, \mathfrak{A}\right)$. To check that $\Psi$ is a $\mathcal{C}^{\infty}$ function at the point $z_{0}$,
we consider two situations:
\begin{enumerate}
  \item [(a)] For all $ 1 \leq  k  \leq  N+1$, $z_{0} \not\in\partial V_{k}^{\varepsilon}$, then there exists a
  open neighborhood $\displaystyle W_{z_{0}}\subset \bigcap_{k \in Ind(z)} V_{k}$ of $z_{0}$ such that, for all $z \in  W_{z_{0}}$:
   $$  \Psi (z)= \sum_{k \in Ind(z_{0})}\mathcal{R}_{k,k}\left[f\right](z).$$
       In this case, $\Psi$ is $\mathcal{C}^{\infty}$ on $W_{z_{0}}$ by property (1) of $\mathcal{R}_{_{\mathfrak{A}}}$,
       (see proposition \ref{propriétés de l'operateur de regularisation}).
  \item [(b)] There exists $ 1 \leq  j  \leq  N+1$ such that $z_{0} \in\partial V_{j}^{\varepsilon}$. In this case, there exists a
   neighborhood $ W_{z_{0}}$ such that for all $z \in W_{z_{0}}$
   \begin{equation*}
   \Psi(z)= \left\{
    \begin{split}
      \sum_{k \in Ind(z_{0})}\mathcal{R}_{k,k}\left[f\right](z) & \phantom{rrr} \textrm{if} \phantom{r} z \in  W_{z_{0}} - V_{j} \\
       \mathcal{R}_{j,j}\left[f\right](z) + \sum_{k \in Ind(z_{0})}\mathcal{R}_{k,k}\left[f\right](z) &
       \phantom{rrr} \textrm{if} \phantom{r} z \in  W_{z_{0}} \bigcap V_{j}.
    \end{split}
    \right.
   \end{equation*}
 Since we have by corollary \ref{derivations du regularisé},
 $\displaystyle\lim_{z \longrightarrow z_{0}} \partial_{z}^{\alpha}\partial_{\overline{z}}^{\beta}\mathcal{R}_{j,j}\left[f\right](z) = 0$, and since
 $ \displaystyle\sum_{k \in Ind(z_{0})}\mathcal{R}_{k,k}\left[f\right](z)$  is $\mathcal{C}^{\infty}$  on $W_{z_{0}}$, then
 $\Psi$ is $\mathcal{C}^{\infty}$
  on $W_{z_{0}}$. By letting $z_{0}$ running over $Y = \left(\mathbb{C}^{n}, \mathfrak{A}\right)$,
  we obtain then from (a) and (b), that $\Psi$ is $\mathcal{C}^{\infty}$ on $Y = \left(\mathbb{C}^{n}, \mathfrak{A}\right)$.
\end{enumerate}
2) Let $z\in Y = \left(\mathbb{C}^{n}, \mathfrak{A}\right)$, and let
\begin{equation*}
    \begin{split}
        \Psi(z) & = trace \left( \mathcal{R}_{\mathfrak{A}}\left[f\right]\right)(z) \\
      & = \varepsilon^{2n}\sum_{k \in Ind(z)}\phantom{r}\int\limits_{\xi \in V_{k}^{2\varepsilon}}f(\xi).
      \chi_{\varepsilon}\left(\widehat{\psi}_{k}(z)-\widehat{\psi}_{k} (\xi) \right)
       d \widehat{\psi}_{k}.
     \end{split}
\end{equation*}
Since the property of plurisubharmonicity is independent of the choice of local coordinates, it suffices then to check the plurisubharmonicity of
 each function  $$\Psi_{k}(z): =  \mathcal{R}_{k,k}^{\varepsilon}\left[f\right](z) =
 \int\limits_{\xi \in V_{k}^{2\varepsilon}}f(\xi).\chi_{\varepsilon}\left(\widehat{\psi}_{k}(z)-\widehat{\psi}_{k} (\xi) \right)
 d \widehat{\psi}_{k}$$ only in the chart $\left(V_{k}, \psi_{k}\right)$. Let then $\widetilde{f}_{k}$ and $\widetilde{\Psi}_{k}$
 be the expressions of $f$ and $\Psi_{k}$ respectively in the local coordinates $\zeta = \psi_{k}(z)$.
 We have by formula (\ref{l'expression de Fj dans la carte Vj})
$$
\widetilde{\Psi}_{k}(\zeta) =
\varepsilon^{2n}\int\limits_{\eta \in \mathcal{B}_{\varepsilon}} \widetilde{f}_{k}(\zeta-\eta).\chi\left( \frac{\eta}{\varepsilon}\right)d\eta.
$$
By differentiating under the sign sum, we obtain
$$
\frac{\partial^{2}\widetilde{\Psi}_{k}}{\partial \zeta_{j}\partial \zeta_{\ell}}(\zeta) =
\varepsilon^{2n}\int\limits_{\eta \in \mathcal{B}_{\varepsilon}}
\frac{\partial^{2}\widetilde{f}_{k}}{\partial \zeta_{j}\partial \zeta_{\ell}}(\zeta-\eta).\chi\left( \frac{\eta}{\varepsilon}\right)d\eta.
$$
which means that for all $\phantom{r}t = \left(t_{1}, ..., t_{n}\right)\in \mathbb{C}^{n}$
\begin{equation}\label{La forme de Levi de phi}
    \mathscr{L}_{\zeta}(\widetilde{\Psi}_{k})\left[t,t\right] =
    \varepsilon^{2n}\int\limits_{\eta \in \mathcal{B}_{\varepsilon}}
    \mathscr{L}_{(\zeta-\eta)}(\widetilde{f}_{k})\left[t,t\right].\chi\left( \frac{\eta}{\varepsilon}\right)d\eta.
\end{equation}
Since by (\ref{La forme de Levi de f}), the Levi form $\mathscr{L}_{(\zeta-\eta)}(\widetilde{f}_{k})\left[t,t\right]$
is positive defined on $\mathbb{K}_{1}$, that is, almost everywhere positive defined, then by (\ref{La forme de Levi de phi}), the Levi form of $\Psi_{k}$
 is positive defined in $V_{k}^{\varepsilon}$. Hence the Levi form of $\Psi$
 is positive defined in the open neighborhood $W_{z}$ of the point $z$
$$W_{z} = \bigcap_{k \in Ind(z)} V_{k}^{\varepsilon}.$$
Therefore the function $\Psi$ which is $\mathcal{C}^{\infty}$ is strictly \textbf{psh} on $W_{z}$.
By letting $z$ running over $Y= \left(\mathbb{C}^{n}, \mathfrak{A}\right)$, we obtain that the function $\Psi$ is strictly \textbf{psh}
on $Y = \left(\mathbb{C}^{n}, \mathfrak{A}\right)$.\\
3) Let $c > 0$, and assume by contradiction that the sublevel set
$$
L_{c}(\Psi): = \bigg\{z\in Y, \phantom{rrrr} \Psi(z) < c    \bigg\}
$$
is not relatively compact. This implies that $L_{c}(\Psi)$ is not bounded. Since the domain $V_{_{N+1}}$ of the distinguished
chart $\left( V_{_{N+1}}, \psi_{_{N+1}}\right)$ is the only one which is non bounded, and since $f_{_{\mathfrak{A}}}$ is already an exhaustion of
$Y = \left( \mathbb{C}^{n}, \mathfrak{A}\right)$, and is continuous on $V_{_{N+1}}$, then there exists a sequence of points $z_{m} \in V_{_{N+1}}$,
such that
$$\lim_{m \longrightarrow +\infty}\left\| z_{m}\right\|  = +\infty$$ and
\begin{equation}\label{1}
    \lim_{m \longrightarrow +\infty}\left\{\min_{\xi \in \mathcal{B}\left( z_{m}, \varepsilon\right)} f_{_{\mathfrak{A}}}(\xi)\right\}  = +\infty
\end{equation}
 and
 \begin{equation}\label{2}
    \textrm{For all} \phantom{r} m \in\mathbb{N},\phantom{rrrr}  \Psi\left(z_{m}\right) < c .
 \end{equation}
Taking into account the definition of $\Psi$, we obtain from (\ref{2}) by an elementary calculus
\begin{equation*}
    \begin{split}
    c & \geq \Psi\left(z_{m}\right)\\
      & \geq \mathcal{R}_{_{N+1,N+1}}\left[ f_{_{\mathfrak{A}}} \right]\left(z_{m}\right) \\
      & \geq  \varepsilon^{2n}\int\limits_{\eta \in \mathcal{B}_{\varepsilon}}\left(f_{_{\mathfrak{A}}}\circ\psi_{_{N+1}}^{-1}\right)\left(\psi_{_{N+1}}
      \left(z_{m}\right)-\eta\right).\chi_{\varepsilon}\left( \eta \right) d\eta \\
      & \geq \min_{\xi \in \mathcal{B}\left( z_{m}, \varepsilon\right)} f_{_{\mathfrak{A}}}(\xi).
      \bigg(\varepsilon^{2n}\int\limits_{\eta \in \mathcal{B}_{\varepsilon}}\chi_{\varepsilon}\left( \eta \right) d\eta\bigg) \\
      & \geq \min_{\xi \in \mathcal{B}\left( z_{m}, \varepsilon\right)} f_{_{\mathfrak{A}}}(\xi).
    \end{split}
\end{equation*}
which means that for all $m \in\in \mathbb{N}$
$$ c  \geq \min_{\xi \in \mathcal{B}\left( z_{m}, \varepsilon\right)} f_{_{\mathfrak{A}}}(\xi).$$
But this contradicts (\ref{1}). The hypothesis that the sublevel set $L_{c}(\Psi)$ is not relatively compact, is then false.\\
The previous properties of $\Psi$ mean that $\Psi$ is an exhaustive $\mathcal{C}^{\infty}$ strictly \textbf{psh} function
on $\left( \mathbb{C}^{n}, \mathfrak{A} \right)$. The hypothetic complex manifold
$Y = \left( \mathbb{C}^{n} , \mathfrak{A}\right)$ is then strictly pseudoconvex. The proof of theorem
\ref{La variete Y est pseudoconvexe} is then complete.
\end{proof}
\section{\textbf{Proof of the main theorem}}
We are now in a position to prove the main theorem \ref{equiv}.
We have already observed in remark \ref{remarque apres equiv} that $\{Hath\}_{1} \Longleftrightarrow \{Necs\}_{1}$ holds, and then we have to prove
theorem \ref{equiv} only for $n \geq 2$.
We begin first by proving the implication $$\{Hath\}_{n} \Longrightarrow \{Necs\}_{n}$$ then we prove the converse
$$\{Necs\}_{n} \Longrightarrow\{Hath\}_{n}.$$
\subsection{\textbf{$\left\{Hath\right\}_{n} \Longrightarrow \left\{Necs\right\}_{n}$}}
\begin{proof}$ $\\
For $n \geq 2$, we prove $\{Hath\}_{n} \Longrightarrow \{Necs\}_{n}$  by contradiction, that is, by showing that if $\{Necs\}_{n}$ is false, which
means that the sphere $S_{2n}$ admits a complex analytic structure, then $\{Hath\}_{n}$ implies that the homeomorphisms defining the charts of this
structure are constant, which contradicts the fact that the change of charts is biholomorphic.
The proof will be done in three steps.\\
\textbf{\underline{$1^{st}$step.}}$ $ \\
Assume that proposition $\{Necs\}_{n}$ is false, and let the hypothetic complex atlas
$$\mathcal{A} = \bigg\{\left(U_{\jmath},\varphi_{\jmath} \right),\phantom{rrr} \jmath\in \mathcal{J}\bigg\}$$
of $\mathcal{S}_{2n}$ as considered in subsubsection \ref{L'atlas A}.  In this first step, we want to construct (without changing the
hypothetic complex structure of $S_{2n}$) an appropriate complex atlas $\mathcal{F}$ of $S_{2n}$ finer than $\mathcal{A}$
 $$
 \mathcal{F}= \bigg\{\left(W_{\gamma},\Phi_{\gamma} \right),\phantom{rrr} \gamma\in \mathcal{L}\bigg\}
 $$
 and to attach to each pairing of different charts $\big(\left(W_{\alpha},\Phi_{\alpha}\right),\left(W_{\beta},\Phi_{\beta}\right)\big)\in \mathcal{F}^{2}$
 satisfying the condition
  \begin{equation}\label{condition sur les cartes V alpha et V beta}
    W_{\alpha}\bigcap W_{\beta} \neq \emptyset
 \end{equation}
an $\mathcal{A}-$submanifold
$X_{\alpha,\beta}$ of $\mathcal{S}_{2n}$. The new atlas $\mathcal{F}$ will be constructed as follows: \\
First, recover each ball $\varphi_{\jmath}\left( U_{\jmath}\right)\subset \mathbb{C}^{n}$ by all smaller balls $\mathcal{B}^{\jmath}_{(a,r)}$ which
are centered at points $a \in \varphi_{\jmath}\left(U_{\jmath} \right)$ and which have radii  $r > 0$, such that $\mathcal{B}^{\jmath}_{(a,r)}
\subset\subset \varphi_{\jmath}\left( U_{\jmath}\right).$ Every ball $\varphi_{\jmath}\left( U_{\jmath}\right)$ becomes then union of smaller balls
$\mathcal{B}^{\jmath}_{(a,r)}$
$$\varphi_{\jmath}\left( U_{\jmath}\right) = \displaystyle\bigcup_{(a,r)} \mathcal{B}^{\jmath}_{(a,r)}.$$
Observe that the new family of balls $\mathcal{B}^{\jmath}_{(a,r)}$ is indexed by the following  set
$$
\mathcal{L}:= \bigg\{(\jmath,a,r), \phantom{rr} \textrm{such that} \phantom{r}\jmath\in \mathcal{J}, \phantom{r}a \in
\varphi_{\jmath}\left( U_{\jmath}\right),\phantom{r} r>0 ,\phantom{r}\textrm{with}\phantom{r} B^{\jmath}_{(a,r)}
\subset\subset \varphi_{\jmath}\left( U_{\jmath}\right)\bigg\}
$$
and that by the condition $B^{\jmath}_{(a,r)}\subset \subset \varphi_{\jmath}\left( U_{\jmath}\right)$, the set $\mathcal{L}$ is not
 a cartesian product.\\
We define then the new complex charts  $\bigg(W_{(\jmath,a,r)} , \Phi_{(\jmath,a,r)} \bigg)$ of the sphere $S_{2n}$,
by setting for every ball $\mathcal{B}^{\jmath}_{(a,r)}\subset\subset \varphi_{\jmath}\left( U_{\jmath}\right)$:
\begin{equation*}
\left\{
    \begin{split}
      W_{(\jmath,a,r)} & =  \varphi_{\jmath}^{-1}\left( \mathcal{B}^{\jmath}_{(a,r)}\right)\\
      \Phi_{(\jmath,a,r)}  & = \varphi_{\jmath}\big/W_{(\jmath,a,r)}.
    \end{split}
    \right.
\end{equation*}
Hence, we obtain a new complex atlas $\mathcal{F}$ of $S_{2n}$ finer than $\mathcal{A}$
 $$\mathcal{F}:= \bigg\{\left(W_{(\jmath,a,r)},\Phi_{(\jmath,a,r)} \right),\phantom{rrr} (\jmath,a,r)\in \mathcal{L}\bigg\}.$$
Observe that by construction of the charts $ \left(W_{(\jmath,a,r)}, \Phi_{(\jmath,a,r)}\right)$, we have $U_{\jmath} =
\displaystyle\bigcup_{(a,r)} W_{(\jmath,a,r)} $.\\
\textbf{\underline{Notations}.}
To simplify the notations, we note every triplet $(\jmath,a,r) \in \mathcal{L}$ by one Greek character, which enables us to write the new
atlas $\mathcal{F}$
simply as:
$$\mathcal{F}= \bigg\{\big(W_{\gamma},\Phi_{\gamma} \big),\phantom{rrr} \gamma\in \mathcal{L}\bigg\}.$$
Now fix a pairing of charts $\big(\left(W_{\alpha},\Phi_{\alpha}\right),
\left(W_{\beta},\Phi_{\beta}\right)\big)$
satisfying condition (\ref{condition sur les cartes V alpha et V beta}), that is
 $$ W_{\alpha}\bigcap W_{\beta}\neq \phi$$
and collect all other charts $\big(W_{\gamma},\Phi_{\gamma}\big) \in \mathcal{F}$ whose domain $W_{\gamma}$ does not intersect both
$W_{\alpha}$ and $W_{\beta}$, that is, such that
  \begin{equation}\label{conditions du recouvrement}
    W_{\gamma}\bigcap W_{\alpha}\bigcap W_{\beta} = \emptyset.
  \end{equation}
We obtain then a sub-atlas $\mathcal{F}_{\alpha, \beta} \subset \mathcal{F}$ defined by
$$
\mathcal{F}_{\alpha, \beta}:=  \bigg\{ \left(W_{{\alpha}},\Phi_{\alpha}\right) , \left(W_{\beta},\Phi_{\beta}\right)\bigg\}
\bigcup_{\gamma \neq \alpha,\beta}\bigg\{ \left(W_{\gamma},\Phi_{\gamma} \right) \in \mathcal{F},\phantom{rr}\textrm{with} \phantom{r}W_{\gamma}
\phantom{r}
\textrm{satisfying} \phantom{r} (\ref{conditions du recouvrement})\bigg\}.
$$
Let
\begin{equation}\label{variété X_alpha,beta}
X_{\alpha,\beta}:= \bigcup_{\left(W_{\gamma}, \Phi_{\gamma} \right)\in \mathcal{F}_{\alpha,\beta}} W_{\gamma}.
\end{equation}
The open set $X_{\alpha,\beta}$ is then an open complex submanifold of $\mathcal{S}_{2n}$ whose complex structure is
defined by the subatlas $\mathcal{F}_{\alpha,\beta}$, and $X_{\alpha,\beta}$ is covered by
the following family of open sets
\begin{equation}\label{le recouvrement}
    \mathfrak{V}_{\alpha,\beta}: =  \bigg\{ W_{\gamma},\phantom{rrrr}\textrm{such that}\phantom{r}\left(W_{\gamma},\Phi_{\gamma}\right)\in
     \mathcal{F}_{\alpha,\beta}\bigg\}.
\end{equation}
By condition (\ref{conditions du recouvrement}), the complementary of $X_{\alpha,\beta}$ with respect to $\mathcal{S}_{2n}$ is
$$\mathcal{S}_{2n}-X_{\alpha,\beta} = \partial W_{\alpha}\cap \partial W_{\beta}$$
and then
\begin{equation}\label{(a)}
    X_{\alpha,\beta} = \mathcal{S}_{2n}- \mathcal{K}
\end{equation}
where the compact $\mathcal{K}$ is given by \phantom{r}
$$\mathcal{K} =  \partial W_{\alpha}\cap \partial W_{\beta} \approx \mathcal{S}_{2n-2}.$$
By writing $\alpha = (j,a,r)\in \mathcal{L}$ and $\beta = (k,b,s)\in \mathcal{L}$, we observe that
\begin{equation}\label{(b)}
    \mathcal{K} \subset U_{\jmath} \bigcap U_{k} \subset U_{\jmath}
\end{equation}
and that
\begin{equation}\label{(c)}
    \varphi_{\jmath}(\mathcal{K}) = \mathcal{S}_{2n-2}(b,r).
\end{equation}
Equations (\ref{(a)}) , (\ref{(b)}) , (\ref{(c)}) imply that the open complex manifold $X_{\alpha,\beta}$ is an
$\mathcal{A}-$submanifold of $\mathcal{S}_{2n}$
in the sense of definition \ref{A-sous variete de S_2n}.\\
\textbf{\underline{$2^{nd}-$step}.}$ $
 Our goal in this step, is to construct two \textbf{holomorphic} mappings
$$
F : X_{\alpha,\beta}-\overline{W_{\alpha}} \longrightarrow \mathbb{C}^{n}
$$
and
$$
G : X_{\alpha,\beta}-\overline{W_{\beta}} \longrightarrow \mathbb{C}^{n}
$$
such that
$$
G\big/\left(W_{\alpha} \cap W_{\beta}\right) - F\big/\left( W_{\alpha}\cap W_{\beta}\right) =  \Phi_{\beta} - \Phi_{\alpha}\phantom{rrrrrr} \textrm{on} \phantom{rr} W_{\alpha} \cap W_{\beta}.
$$
Indeed, let $\mathcal{O}^{n}$ be the sheaf of holomorphic mappings on $\mathcal{S}_{2n}$ with values in $\mathbb{C}^{n}$, and let $\mathcal{O}^{n}_{\alpha,\beta}$ denote the restriction of the sheaf $\mathcal{O}^{n}$ to the $\mathcal{A}-$submanifold $X_{\alpha,\beta}$, that is
$$
\mathcal{O}^{n}_{\alpha,\beta}: = \mathcal{O}^{n}\big/X_{\alpha,\beta}.
$$
Consider with respect to the open covering $\mathfrak{V}_{\alpha,\beta}$ of $X_{\alpha,\beta}$, the following groups of cochains with values in the sheaf
 $\mathcal{O}^{n}_{\alpha,\beta}$:
$$
C^{0}\bigg( \mathfrak{V}_{\alpha,\beta} , \mathcal{O}^{n}_{\alpha,\beta} \bigg) = \prod_{\gamma\in \mathcal{L}}\mathcal{O}^{n}_{\alpha,\beta}\big(W_{\gamma}  \big)
$$
$$
C^{1}\bigg(  \mathfrak{V}_{\alpha,\beta} , \mathcal{O}^{n}_{\alpha,\beta} \bigg)= \prod_{\left(\gamma,\mu\right)\in \mathcal{L}^{2}}
\mathcal{O}^{n}_{\alpha,\beta}\left(W_{\gamma}\cap W_{\mu} \right)
$$
$$
C^{2}\bigg( \mathfrak{V}_{\alpha,\beta} , \mathcal{O}^{n}_{\alpha,\beta} \bigg)= \prod_{\left(\gamma,\mu,\nu\right)\in \mathcal{L}^{3}}
\mathcal{O}^{n}_{\alpha,\beta}\left(W_{\gamma}\cap W_{\mu}\cap W_{\nu}\right)
$$
and let the coboundary operators:
$$
\delta: C^{0}\bigg( \mathfrak{V}_{\alpha,\beta} , \mathcal{O}^{n}_{\alpha,\beta} \bigg)  \longrightarrow C^{1}\bigg( \mathfrak{V}_{\alpha,\beta} , \mathcal{O}^{n}_{\alpha,\beta} \bigg)
$$
and
$$
\delta: C^{1}\bigg( \mathfrak{V}_{\alpha,\beta} , \mathcal{O}^{n}_{\alpha,\beta} \bigg)  \longrightarrow C^{1}\bigg( \mathfrak{V}_{\alpha,\beta} , \mathcal{O}^{n}_{\alpha,\beta} \bigg)
$$
defined by:\\
1) For the 0-cochains $\big(f_{\gamma}\big)_{(\gamma)\in \mathcal{L}} \in C^{0}\bigg( \mathfrak{V}_{\alpha,\beta} , \mathcal{O}^{n}_{\alpha,\beta} \bigg)$
$$\delta \bigg(\big(f_{\gamma}\big)_{(\gamma)\in \mathcal{L}} \bigg) = \big(g_{\gamma,\mu}\big)_{(\gamma,\mu)\in \mathcal{L}^{2}}$$ where
$$
g_{\gamma,\mu} = f_{\mu} - f_{\gamma} \phantom{rrrr} \textrm{on} \phantom{r}W_{\gamma}\cap W_{\mu}.
$$
2) For the 1-cochains $\big(f_{\gamma,\mu}\big)_{(\gamma,\mu)\in \mathcal{L}^{2}}\in C^{1}\bigg( \mathfrak{V}_{\alpha,\beta}, \mathcal{O}^{n}_{\alpha,\beta} \bigg)$ $$\delta \bigg(\big(f_{\gamma,\mu}\big)_{(\gamma,\mu)\in \mathcal{L}^{2}}\bigg)
 = \big(g_{\gamma,\mu,\nu}\big)_{(\gamma,\mu,\nu)\in \mathcal{L}^{3}} $$ where
$$
g_{\gamma,\mu,\nu} = f_{\mu,\nu}-f_{\gamma,\nu}+f_{\gamma,\mu}\phantom{rrrr} \textrm{on} \phantom{r}W_{\gamma}\cap W_{\mu}\cap W_{\nu}.
$$
Since the boundary operators are trivially group homomorphisms, we obtain then the group of 1-cycles
$$
\mathcal{Z}^{1}\bigg( \mathfrak{V}_{\alpha,\beta} , \mathcal{O}^{n}_{\alpha,\beta} \bigg): = ker\phantom{r} C^{1}\bigg(  \mathfrak{V}_{\alpha,\beta} , \mathcal{O}^{n}_{\alpha,\beta} \bigg) \stackrel{\delta}{\longrightarrow}  C^{2}\bigg(  \mathfrak{V}_{\alpha,\beta}, \mathcal{O}^{n}_{\alpha,\beta} \bigg)
$$
and the group of 1-coboundaries
$$
\mathcal{B}^{1}\bigg(  \mathfrak{V}_{\alpha,\beta}, \mathcal{O}^{n}_{\alpha,\beta} \bigg): = Im \phantom{r} C^{0}\bigg( \mathfrak{V}_{\alpha,\beta}, \mathcal{O}^{n}_{\alpha,\beta} \bigg) \stackrel{\delta}{\longrightarrow}  C^{1}\bigg( \mathfrak{V}_{\alpha,\beta}, \mathcal{O}^{n}_{\alpha,\beta} \bigg)
$$
and the cohomology group with coefficients in the sheaf $\mathcal{O}^{n}_{\alpha,\beta}$, with respect to the covering $\mathfrak{V}_{\alpha,\beta}$ is
$$
\mathcal{H}^{1}\bigg( \mathfrak{V}_{\alpha,\beta}, \mathcal{O}^{n}_{\alpha,\beta} \bigg): =
\mathcal{Z}^{1}\bigg(  \mathfrak{V}_{\alpha,\beta} , \mathcal{O}^{n}_{\alpha,\beta} \bigg)\bigg/
\mathcal{B}^{1}\bigg(  \mathfrak{V}_{\alpha,\beta}, \mathcal{O}^{n}_{\alpha,\beta} \bigg).
$$
Consider now on the $\mathcal{A}-$submanifold  $X_{\alpha,\beta}$ the following $1-$cochain $\big(f_{\gamma,\mu}\big)_{(\gamma,\mu)\in \mathcal{L}^{2}}$ 
 with $f_{\gamma,\mu} \in \mathcal{O}^{n}_{\alpha,\beta}\left(W_{\gamma}\cap W_{\mu} \right)$ defined by:
\begin{equation}\label{cochain}
\left\{
   \begin{split}
     f_{\alpha,\beta} & = \Phi_{\beta} - \Phi_{\alpha}  \\
      f_{\gamma,\mu} & = 0 \phantom{rrrr} \textrm{for all}\phantom{r} (\gamma,\mu)\neq(\alpha,\beta).
   \end{split}
   \right.
\end{equation}
 Since the homeomorphism $\Phi_{\alpha}$ and $\Phi_{\beta}$ are holomorphic by the hypothetic complex structure of the sphere $\mathcal{S}_{2n}$, 
 then the $1-$cochain $\big(f_{\gamma,\mu}\big)_{(\gamma,\mu)\in \mathcal{L}^{2}}$ takes its values in the sheaf $\mathcal{O}^{n}_{\alpha,\beta}$, 
 and is therefore well defined. \\We claim that the 1-chain $\big(f_{\gamma,\mu}\big)_{(\gamma,\mu)\in \mathcal{L}^{2}}$ is a 1-coboundary. 
 To prove this,  let's first check that $\big(f_{\gamma,\mu}\big)_{(\gamma,\mu)\in \mathcal{L}^{2}}$ is a 1-cocycle. Indeed, since by condition (\ref{conditions du recouvrement}), the covering $\mathfrak{V}_{\alpha,\beta}$  does not contain any domain $W_{\gamma}$ which intersects both $W_{\alpha}$ and $W_{\beta}$, then for every triplet of domains $W_{\gamma}, W_{\mu},  W_{\nu}$ of $X_{\alpha,\beta}$, satisfying
$$W_{\gamma} \cap W_{\mu}\cap  W_{\nu} \neq \phi$$
we have trivially  $$\delta \bigg(\big(f_{\gamma,\mu}\big)_{(\gamma,\mu)\in \mathcal{L}^{2}}\bigg) = 0$$
which means that the 1-chain $\big(f_{\gamma,\mu}\big)_{(\gamma,\mu)\in \mathcal{L}^{2}}$ is a 1-cocycle.\\
Furthermore, we know by the vanishing theorem \ref{theoreme d'annulation}, that:
$$\mathcal{H}^{0,1}\left(X_{\alpha,\beta}, \mathbb{C}^{n}\right) = \left\{  0 \right\},$$
which implies by Dolbeault's resolution of the sheaf $\mathcal{O}^{n}_{\alpha,\beta} $, that
$$\mathcal{H}^{1}\bigg( X_{\alpha,\beta}, \mathcal{O}^{n}_{\alpha,\beta} \bigg) = \left\{  0\right\},$$
 and then according to Leray's theorem \ref{Leray} (see \ref{Annexe} index), the cohomology group with respect to the covering $\mathfrak{V}_{\alpha,\beta}$ is trivial, that is
$$
\mathcal{H}^{1}\bigg( \mathfrak{V}_{\alpha,\beta}, \mathcal{O}^{n}_{\alpha,\beta} \bigg) = \left\{  0\right\}.
$$
Thus the 1-cochain $\big(f_{\gamma,\mu}\big)_{(\gamma,\mu)\in \mathcal{L}^{2}}$ is a 1-boundary.\\
This means that there exists over $X_{\alpha,\beta}$ a 0-cochain
$
\big(g_{\gamma}\big)_{\gamma\in \mathcal{L}}
$
with values in the sheaf $\mathcal{O}^{n}_{\alpha,\beta}$,  such that
$$
\delta \bigg( \big(g_{\gamma}\big)_{\gamma\in \mathcal{L}} \bigg) = \big(f_{\gamma,\mu}\big)_{(\gamma,\mu)\in \mathcal{L}^{2}}
$$
that is
$$
 g_{\mu} - g_{\nu} = f_{\mu,\nu} \phantom{rrrrrrrrr} \textrm{on} \phantom{rr} W_{\mu}\bigcap W_{\nu}.
$$
Observe from the definition of the 1-chain $\left(f_{\mu,\nu} \right)_{(\mu,\nu)\in \mathcal{L}^{2}}$, that the restrictions of the 0-chain $\big(g_{\gamma}\big)_{\gamma\in \mathcal{L}}$ to the open sets
$X_{\alpha,\beta} - \overline{W_{\alpha}}$ and $X_{\alpha,\beta} - \overline{W_{\beta}}$, satisfy the following compatibility conditions:
\begin{enumerate}
  \item On the open set $X_{\alpha,\beta}-\overline{ W_{\alpha}}$, we have
  $$
  g_{\mu} - g_{\nu} = 0 \phantom{rrrrrrrrr} \textrm{on} \phantom{rr} W_{\mu}\bigcap W_{\nu}
  $$
 \item On the open set $X_{\alpha,\beta} - \overline{W_{\beta}}$, we have
  $$
  g_{\mu} - g_{\nu} = 0 \phantom{rrrrrrrrr} \textrm{on} \phantom{rr} W_{\mu}\bigcap W_{\nu}.
  $$
  \end{enumerate}
Then there exist two holomorphic mappings with values in $\mathbb{C}^{n}$:
$$F: X_{\alpha,\beta}- \overline{W_{\alpha}} \longrightarrow \mathbb{C}^{n}$$
and
  $$G: X_{\alpha,\beta}- \overline{W_{\beta}} \longrightarrow \mathbb{C}^{n}$$
such that
  $$
  F/ W_{\mu} = g_{\mu}\phantom{rrrr} \textrm{on} \phantom{rrrrrr} W_{\mu}
  $$
  and
  $$
  G/ W_{\nu} = g_{\nu}\phantom{rrrr} \textrm{on} \phantom{rrrrrr} W_{\nu}
  $$
that is, for $\mu = \alpha$ and $\nu = \beta$
\begin{equation}\label{fonction - fonction}
\begin{split}
    F\big/\left(W_{\alpha}\cap W_{\beta}\right) -   G\big/\left(W_{\alpha}\cap W_{\beta}\right) & = g_{\alpha}/\left(W_{\alpha}\cap W_{\beta}\right) -
    g_{\beta}/\left(W_{\alpha}\cap W_{\beta}\right) \\
    & = f_{\alpha,\beta}/\left(W_{\alpha}\cap W_{\beta}\right) \\
    & = \Phi_{\alpha}\left(W_{\alpha}\cap W_{\beta}\right) - \Phi_{\beta}\left(W_{\alpha}\cap W_{\beta}\right).
    \end{split}
\end{equation}
\textbf{\underline{$3^{rd}$step.}} In this third step, we are going to deduce that for $n \geq 2$, $\left\{\mathcal{H}ath \right\}_{n}$ implies that the homeomorphisms $\Phi_{\alpha}$ defining the charts of the finer atlas $\mathcal{F}$ are constant.
Indeed, assume that for $n \geq 2$, proposition $\left\{ Hath\right\}_{n}$ is true, and
let the complex chart $\left(U_{\jmath} , \Phi_{\jmath}\right)$ of the hypothetic atlas $\mathcal{A}$, (recall that $U_{\jmath}$ is assumed to be homeomorphic to a ball of $\mathbb{C}^{n}$). Observe that the open set $\left(X_{\alpha,\beta} - \overline{W_{\alpha}}\right) \bigcap U_{i}$ is connected.  Since the restricted mapping
$$F:\left( X_{\alpha,\beta}- \overline{W_{\alpha}}\right) \bigcap U_{\jmath}\longrightarrow \mathbb{C}^{n}$$
is holomorphic, it follows then by lemma \ref{Hartogs local}, that $F/_{(X_{\alpha,\beta}- \overline{W_{\alpha}} )\bigcap U_{\jmath}}$ admits an analytic continuation to the whole domain $U_{\jmath}$.  Hence the holomorphic mapping $$F: X_{\alpha,\beta}- \overline{W_{\alpha}} \longrightarrow \mathbb{C}^{n}$$ admits an analytic continuation to the whole sphere $\mathcal{S}_{2n}$ (still denoted $F$)
  $$
  F : \mathcal{S}_{2n}  \longrightarrow  \mathbb{C}^{n}.
  $$
  But, since $\mathcal{S}_{2n}$ is compact, then $F$ must be constant on the sphere $\mathcal{S}_{2n}$. By a similar argument, we show that the holomorphic mapping $$G: X_{\alpha,\beta}- \overline{W_{\beta}} \longrightarrow \mathbb{C}^{n}$$
  admits an analytic continuation to the whole sphere $\mathcal{S}_{2n}$, and then becomes also constant on $\mathcal{S}_{2n}$. \\
Taking into account the fact that $F$ and $G$ are both constant mappings on $\mathcal{S}_{2n}$, we obtain then by differentiating (\ref{fonction - fonction}) on $\left(W_{\alpha}\cap W_{\beta}\right)$
\begin{equation}\label{global one form}
    \partial \Phi_{\alpha} = \partial \Phi_{\beta}  \phantom{rrrrrrrrr} \textrm{on} \phantom{rr} W_{\alpha}\cap W_{\beta}.
\end{equation}
Recall that until now, the indexes $\alpha $ and $\beta$ $\in \mathcal{L}$ are assumed to be fixed, but when we let $\alpha$ and $\beta$ running over the index set $\mathcal{L}$, the identity (\ref{global one form}) becomes then valid for all $\alpha,\beta \in \mathcal{L}$, that is
\begin{equation}\label{global one form 2}
    \partial \Phi_{\alpha} = \partial \Phi_{\beta}  \phantom{rrrrr} \textrm{n} \phantom{rr} W_{\alpha}\cap W_{\beta}.
\end{equation}
Hence, there exists by the compatibility conditions (\ref{global one form 2}), a global 1-differential vectorial form $\omega$  of type (1,0) on the whole sphere $\mathcal{S}_{2n}$, with holomorphic coefficients, such that
\begin{equation}\label{existence de omega}
    \omega/ W_{\gamma} = \partial \Phi_{\gamma} \phantom{rrrrrrrrr} \textrm{for all} \phantom{rr} \gamma \in \mathcal{L}.
\end{equation}
By differentiating (\ref{existence de omega}) we obtain trivially
$$\partial \omega = 0 \phantom{rrrrrrr} \textrm{on} \phantom{r} \mathcal{S}_{2n}.$$
Since by the vanishing theorem \ref{theoreme d'annulation}, the cohomology group $\mathcal{H}^{1,0}\left(X_{\alpha,\beta}, \mathbb{C}^{n} \right)$ is also trivial
$$\mathcal{H}^{1,0}\left(X_{\alpha,\beta}, \mathbb{C}^{n} \right) = \left\{ 0\right\}$$
there exists then a $\mathcal{C}^{\infty}$ mapping $H: X_{\alpha,\beta} \longrightarrow \mathbb{C}^{n}$ such that
\begin{equation}\label{omega est exact}
    \partial H = \omega.
\end{equation}
Observing that $\omega$ has holomorphic coefficients, and that $\partial H = \omega $, we deduce that $H : X_{\alpha,\beta} \longrightarrow \mathbb{C}^{n}$, is a harmonic mapping and that $H$ can be written in the form
$$
H = H_{1} +\overline{ H_{2}}
$$
where both mappings $$H_{1} , H_{2}: X_{\alpha,\beta} \longrightarrow \mathbb{C}^{n}$$ are both holomorphic on $X_{\alpha,\beta}$, and where $\overline{H_{2}}$ denotes the conjugate of $H_{2}$.\\Using once again the same arguments as for the holomorphic mappings $F$ and $G$ above, one can prove that both holomorphic mappings $H_{1}$ and $H_{2}$ which are already holomorphic on $X_{\alpha,\beta}$, admit analytic continuations to the whole sphere $\mathcal{S}_{2n}$. Both mappings $H_{1}$ and $H_{2}$ become then constant on $\mathcal{S}_{2n}$. According to equation (\ref{omega est exact}) we obtain $\omega = 0$, and thanks to equation (\ref{existence de omega}), we can finally conclude that:
$$
\textrm{for all} \phantom{r} \gamma \in \mathcal{L},\phantom{rrrrr}\Phi_{\gamma}\phantom{r} \textrm{is constant on} \phantom{r} W_{\gamma} .
$$
But this last conclusion contradicts the fact that $\left( W_{\gamma}, \Phi_{\gamma} \right)$ is a chart of a complex structure on the sphere $S_{2n}$.
Then the assumption that $\left\{Necs\right\}_{n}$ is false for $n \geq 2$, has lead to a contradiction. The implication $ \left\{Hath\right\}_{n}  \Longrightarrow \left\{Necs\right\}_{n}$  is then true for all $n  \geq 2$.
\end{proof}
\subsection{$\left\{Necs\right\}_{n}\Longrightarrow \left\{Hath\right\}_{n}$ }$ $\\
to prove the implication $\left\{ Necs \right\}_{n} \Longrightarrow  \left\{ Hath \right\}_{n}$, we need first to prove the following lemma.
\begin{lemma}\label{reciproque}$ $\\
Let the unit ball  $\mathcal{B}_{1}$ of $\mathbb{C}^{n}$, and let $\mathcal{K}$ be a compact subset of $\mathcal{B}_{1}$
such that $\mathcal{B}_{1}-\mathcal{K}$ is connected. Assume that $\left\{ Hath \right\}_{n}$ is false, that is, there exists a holomorphic
function $f: \mathcal{B}_{1}-  \mathcal{K} \longrightarrow \mathbb{C}$ which does not admit an analytic continuation to the ball $\mathcal{B}_{1}$.
Then there exist two balls $\mathcal{B}_{r_{1}}$ and $\mathcal{B}_{r_{2}}$ with
$$
\mathcal{K} \subset \mathcal{B}_{r_{1}} \subset \mathcal{B}_{r_{2}} \subset \mathcal{B}_{1}
$$
 and there exist two open convex sets $\mathcal{C}_{1}$ and
$\mathcal{C}_{2}$ with
$$
\mathcal{K} \subset  \mathcal{C}_{1}   \subset  \mathcal{C}_{2} \subset \mathcal{B}_{1}
$$
and there exists a
 bi-holomorphic mapping $$F: \mathcal{C}_{2}- \mathcal{C}_{1} \longrightarrow \mathcal{B}_{r_{2}} - \mathcal{B}_{r_{1}}$$
 which does not admit an analytic continuation to the open convex set $\mathcal{C}_{2}$.
\end{lemma}
\begin{proof}$ $\\
For $ 0 < r < r' < 1$, Let $\mathcal{B}_{r}$ and $\mathcal{B}_{r'}$ be two balls of $\mathbb{C}^{n}$ such that
$$\mathcal{K} \subset \mathcal{B}_{r} \subset \mathcal{B}_{r'} \subset \mathcal{B}_{1}$$ and
observe that the annulus $\mathcal{K}_{r,r'}:= \overline{\mathcal{B}}_{r'}- \mathcal{B}_{r}$ is a compact subset of $\mathcal{B}_{1}-\mathcal{K}$.\\
Let the holomorphic function $f: \mathcal{B}_{1}-\mathcal{K}\longrightarrow \mathbb{C}$ satisfying the hypothesis of the lemma, and for a positive
number $\lambda > 0$, let the mapping
$$
F_{\lambda}: \mathcal{B}_{r'} - \mathcal {B}_{r}\longrightarrow F_{\lambda}(\mathcal{B}_{r'} - \mathcal{B}_{r} )\subset\mathbb{C}^{n}
$$
defined for $z = \left(z_{1}, ..., z_{n}\right) \in  \mathcal{B}_{r'} - \mathcal {B}_{r}$ by:
\begin{equation}\label{Le bi holomorphisme F lambda}
F_{\lambda}(z) = \bigg(z_{1} + \frac{f\left(z_{1},z_{2},...,z_{n}\right)}{\lambda},  z_{2},..., z_{n} \bigg).
\end{equation}
It is clear that $F_{\lambda}$ is holomorphic in $\mathcal{B}_{r'}- \mathcal{B}_{r}$ and doesn't admit an analytic continuation to $\mathcal{B}_{r'}$, and
that $f$ and its derivatives $\displaystyle\frac{\partial f}{\partial z_{j}}$ and $\displaystyle\frac{\partial^{2}f}{\partial z_{j}\partial z_{k}}$ are
bounded on the compact annulus $\mathcal{K}_{r,r'}$. Let:
$$
    \theta = \sup_{z \in \mathcal{K}_{r,r'}}\left|f(z)\right|
$$
$$
    \theta_{j} = \sup_{z \in \mathcal{K}_{r,r'}}\left|\frac{\partial f}{\partial z_{j}}(z)\right|
$$
$$
    \theta_{j,k} = \sup_{z \in \mathcal{K}_{r,r'}}\left|\frac{\partial^{2} f}{\partial z_{j}\partial z_{k}}(z)\right|
$$
and
$$
\widetilde{\theta} = \max_{1 \leq j,k \leq n} \left\{\theta , \theta_{j}, \theta_{j,k}\right\}.
$$
The $ n \times n$-matrix
$$
A(z)
=
\left(
  \begin{array}{ccccc}
  \frac{\partial f}{\partial z_{1}} &   \frac{\partial f}{\partial z_{2}}&\frac{\partial f}{\partial z_{3}} & \cdots &  \frac{\partial f}{\partial z_{n}} \\
    0 & 0 & 0 & \cdots & 0 \\
    0 & 0 & \ddots & \cdots & 0 \\
    \vdots & \vdots &  & \ddots & \vdots \\
    0 & 0 & 0 & \cdots & 0 \\
  \end{array}
\right).
$$
is then bounded on the compact annulus $\mathcal{K}_{r,r'}$
$$
\| A(z)\| \leq \widetilde{\theta}.
$$
Now let the jacobian matrix of $F_{\lambda}(z)$
$$
Jac(F_{\lambda})(z) = \mathbb{I}_{n} + A(z).
$$
If we choose $\lambda > n.\widetilde{\theta}$, the jacobian $Jac(F_{\lambda})(z)$ is then invertible, and therefore the mapping
 $$
 F_{\lambda}: \mathcal{B}_{r'} -\mathcal{B}_{r}\longrightarrow F_{\lambda}\left(\mathcal{B}_{r'} - \mathcal{B}_{r}\right)\subset\mathbb{C}^{n}
$$
 is bi-holomorphic.\\
From the definition (\ref{Le bi holomorphisme F lambda}) of $F_{\lambda}$, we have for all $z \in \mathcal{B}_{r'} - \mathcal{B}_{r}$
$$
\lim_{\lambda \longrightarrow +\infty}F_{\lambda}(z) = z.
$$
This means that for all $r_{1}, r_{2} > 0$, with $ r < r_{1} < r_{2} < r'$, there exists a positive number $\lambda_{0} > 0$, such that
\begin{equation}\label{existence de r1 et r2}
    \lambda > \lambda_{0} \Longrightarrow \mathcal{B}_{r_{2}} - \mathcal{B}_{r_{1}} \subset F_{\lambda}\left( \mathcal{B}_{1}-\mathcal{K} \right).
\end{equation}
It follows from (\ref{existence de r1 et r2}), that for $\lambda > \lambda_{0}$, there exist two open subsets $\mathcal{C}_{1}$ and
$\mathcal{C}_{2}$, of $\mathcal{B}_{1}$ with
$$ \mathcal{K} \subset \mathcal{C}_{1} \subset \mathcal{C}_{2} \subset \mathcal{B}_{1} $$
such that:
 $$\partial\mathcal{C}_{1}:= \bigg\{z \in \mathbb{C}^{n}, \phantom{rrr} \| F_{\lambda}(z) \|^{2} = r_{1} \bigg\}$$
$$\partial\mathcal{C}_{2}:= \bigg\{z \in \mathbb{C}^{n}, \phantom{rrr} \|F_{\lambda}(z)  \|^{2} = r_{2} \bigg\}$$
and
$$ F_{\lambda}\left(  \mathcal{C}_{2} -  \mathcal{C}_{1} \right)  )  = \mathcal{B}_{r_{2}}   - \mathcal{B}_{r_{_{1}}}.$$
We claim that both open sets $\mathcal{C}_{1}$ and $\mathcal{C}_{2}$ are convex. \\
Indeed, consider the function
$
\rho_{\lambda}: \mathcal{B}_{1} - \mathbb{K} \longrightarrow \mathbb{R}_{+}
$
defined by
\begin{equation}\label{la fonction rho}
    \rho_{\lambda}(z):= \|  F_{\lambda}(z)\|^{2}.
\end{equation}
By differentiating of (\ref{la fonction rho}),
we obtain at a point $z$ near $\partial\mathcal{C}_{j}$, (with j = 1 or j= 2)
 \begin{equation}\label{hessienne}
    Hess\left(\rho_{\lambda}(z)\right) = \mathbb{I}_{2n} + \mathcal{R}(z) \phantom{rrrrrrrr} z \in \partial\mathcal{C}_{j}
 \end{equation}
where $\mathcal{R}(z)$ is a $2n \times 2n$ real matrix involving the first and second real derivatives of the holomorphic function $f$.
Since the derivatives of $f$ are bounded on the compact annulus $\overline{\mathcal{B}}_{r'} - \mathcal{B}_{r}$ by $\widetilde{\theta}$,
then by choosing $\lambda > max \left\{ \lambda_{0} , 2n\widetilde{\theta} \right\}$,
we obtain
$$ \| \mathcal{R}(z) \| < 1.$$
From (\ref{hessienne}), the hessian matrix $Hess\left(\rho_{\lambda}(z)\right)$ is then positive defined, which means that the open sets $\mathcal{C}_{1}$
 and $\mathcal{C}_{2}$ are convex. The proof of lemma \ref{reciproque} is then complete.
\end{proof}
We are now ready to prove the implication $\left\{Necs\right\}_{n}\Longrightarrow \left\{Hath\right\}_{n}$.
\begin{proof}$ $\\
Let $$\bigg\{e_{1}+ie_{2},...,e_{2n-1}+ie_{2n},e_{2n+1} \bigg\}$$ be an orthnormal basis of $\mathbb{C}^{n}\times \mathbb{R}$.
Without loss of generality, we can assume that $\mathcal{S}_{2n}$ is the unit sphere embedded in $\mathbb{C}^{n}\times \mathbb{R}$,
$S_{2n}\hookrightarrow \mathbb{C}^{n}\times \mathbb{R}$ with equation
$$\left| z_{1}\right|^{2} + ... + \left| z_{n}\right|^{2} + s^{2} = 1.$$
Let $N=(0,...,0,1)$ and $S=(0,...,0,-1)$ be respectively the north and the south poles of $\mathcal{S}_{2n}$,
and let the half spheres
$$ W_{+}: = \bigg\{\left(z_{1},...,z_{n},s\right)\in \mathcal{S}_{2n}, \phantom{rrr}  0 < s \leq 1  \bigg\}$$
and
$$ W_{-}: = \bigg\{ \left(z_{1},...,z_{n},s \right)\in \mathcal{S}_{2n}, \phantom{rrr}  -1 \leq  s <0  \bigg\}$$
and for $ 0 < \sigma  < 1 $, let the following open set of $S_{2n}$
$$ W_{\sigma}: = \bigg\{ \left(z_{1},...,z_{n},s\right)\in S_{2n}, \phantom{rrrr} -\sigma < s < 0  \bigg\}.$$
Recall that the stereographic projection of pole $N$, is the homeomorphism
$$
pr_{N}: S_{2n}-\left\{ N\right\} \longrightarrow \mathbb{C}^{n}
$$
which sends every point $P\in S_{2n}- \left\{N \right\}$ to the point $pr_{N}(P)\in \mathbb{C}^{n} $, which is
the intersection of the line $(N,P)$ with the hyperplane $s=0$. The stereographic projection of a point
$P=\left(z_{1},...,z_{n},s\right) \in S_{2n}-\left\{N \right\}$ is given by
$$
pr_{N}(P) = \frac{1}{1-s}.\bigg\{P - s.e_{2n+1} \bigg\}.
$$
Now assume by contradiction that proposition $\{Hath\}_{n}$ is false.
We claim that the sphere $S_{2n}$ admits a complex atlas $\mathcal{A}$ formed of exactly two complex charts $\left(U_{1},\varphi_{1}\right)$ and
$\left(U_{2},\varphi_{2}\right)$.
Indeed, Since $$pr_{N}(W_{-}) = \mathcal{B}_{1}$$  we can define the first chart $\left(U_{1}, \varphi_{1}\right) \in \mathcal{A}$,
by setting
$$
U_{1}:= W_{-} =  \bigg\{\left(z_{1},...,z_{n},\zeta\right)\in S_{2n}, \phantom{rrr}  -1 \leq \zeta < 0  \bigg\}
$$
and by choosing the homeomorphism $\varphi_{1}$ to be the stereographic projection of pole $N$, restricted to $W_{-}$, that is
$$\varphi _{1}:= pr_{N}: W_{-} \longrightarrow \mathcal{B}_{1}.$$
Observe that to define the first chart, we haven't need to use the assumption that $\left\{Hath\right\}_{n}$ is false,  but
for the construction of the second chart $\left( U_{2}, \varphi_{2}  \right)$, this assumption will play a crucial role.
Indeed, by the assumption that $\left\{Hath\right\}_{n}$ is false, there exist according to lemma \ref{reciproque} a compact subset
$\mathcal{K} \subset \mathcal{B}_{1}$ such that $\mathcal{B}_{1} - \mathcal{K}$ is connected, and there exist two convex sets $\mathcal{C}_{1}$ and
$\mathcal{C}_{2}$ and two balls $\mathcal{B}_{r_{1}}$ and $\mathcal{B}_{r_{2}}$ satisfying the conditions:
$$\mathcal{K} \subset \mathcal{C}_{1}\subseteq \mathcal{C}_{2} \subset \mathcal{B}_{1}$$
and
$$ \mathcal{K} \subset \mathcal{B}_{r_{1}}\subseteq \mathcal{B}_{r_{2}} \subset \mathcal{B}_{1} $$
and there exists a bi-holomorphic mapping
$$
F: \mathcal{C}_{2}- \mathcal{C}_{1}\longrightarrow \mathcal{B}_{r_{2}} - \mathcal{B}_{r_{1}}\subset\mathbb{C}^{n}
$$
which does not admit an analytic continuation to $\mathcal{C}_{2}$.
\begin{remark}
by replacing if necessarily, the mapping $F$ by $\frac{1}{r_{2}}.F$, we can assume that $r_{2}=1$, which means that
the bi-homomorphic mapping $F$ can be defined as follows
$$
F: \mathcal{C}_{2}- \overline{\mathcal{C}_{1}}\longrightarrow \mathcal{B}_{1} - \overline{\mathcal{B}}_{r_{1}}.
$$
\end{remark}
Now choose the domain of the second chart to be $U_{2}:= \overline{W_{+}} \bigcup W_{\sigma}$, and observe that
$$
U_{1}\bigcap U_{2} = W_{\sigma}.
$$
The homeomorphism of the second chart will be defined by a mapping
$$
\varphi_{2}: U_{2}=\overline{ W_{+}}\bigcup W_{\sigma} \longrightarrow \overline{\mathcal{C}_{1}}\bigcup \left(\mathcal{C}_{2}-
 \overline{\mathcal{C}}_{1} \right)
$$
where the restrictions
\begin{equation}\label{premier}
    \varphi_{2}\big/ W_{\sigma}: W_{\sigma}   \longrightarrow \mathcal{C}_{2}-\overline{\mathcal{C}_{1}}
\end{equation}
and
\begin{equation}\label{deuxieme}
     \varphi_{2}\big/\overline{W_{+}}: \overline{W_{+}} \longrightarrow \overline{\mathcal{C}_{1}}
\end{equation}
are homeomorphisms. \\To define the restriction $\varphi_{2}\big/ W_{\sigma}$, that is (\ref{premier}), observe first that if we choose
$$
   \sigma = \frac{1-r_{1}}{r_{1}}
$$
then an elementary calculation, shows that
$$
\varphi_{1}\left( W_{\sigma} \right)  = pr_{N}\left( W_{\sigma} \right) = \mathcal{B}_{1}-  \overline{\mathcal{B}}_{r_{1}}.
$$
Thus, for $P\in W_{\sigma}$, we set:
\begin{equation}\label{premier final}
    \varphi_{2}\left( P \right): = F^{-1}\left( \varphi_{1}(P)\right) \in \mathcal{C}_{2}-\overline{\mathcal{C}_{1}}.
\end{equation}
To define the second restriction $\varphi_{2}\big/ \overline{W_{+}}$, that is (\ref{deuxieme}), we need
to use the fact that $\mathcal{C}_{1}$ is convex, and that
the half sphere $W_{+}$ is geodesically convex.\\
Indeed,\\
(1) Since $W_{+}$ is geodesically convex, there exists a mapping
  $$
 \Gamma: \left[0,\frac{\pi}{2}\right] \times \partial W_{+} \longrightarrow \overline{W_{+}}
  $$
  $$
  \phantom{rrrrrrr} (t ,\xi) \longmapsto \Gamma(t,\xi)
  $$ where
  $\Gamma(t,\xi)$ is the geodesic curve starting from the north pole $N$ at $t= 0$ and arriving to the point $\xi \in \partial W_{+}$
  at $t = \frac{\pi}{2}$, that is $\Gamma(0,\xi) = N$, and $\Gamma(\frac{\pi}{2},\xi) = \xi \in \partial W_{+},$ and then
  For every point $P\in \overline{U_{+}}$,
   there exists a unique pairing
  $$(t,\xi) \in \left[0,\frac{\pi}{2}\right] \times \partial W_{+}$$ such that the point $P$ can be written
\begin{equation}\label{ecriture geodesique}
    P = \Gamma(t,\xi).
\end{equation}
The parameter $t \in \left[0,\frac{\pi}{2} \right]$ represents the geodesic distance from $P$ to $N$.\\
(2) Fix a point $ M \in \mathcal{C}_{1}$. By convexity of $\mathcal{C}_{1}$,
every point $\eta \in\partial \mathcal{C}_{1}$ can be joined to the fixed point $M$ by a linear segment $\left[M,\eta\right]$, and then we can define
the mapping
$$ \Lambda:\left[0,\frac{\pi}{2}\right] \times \partial \mathcal{C}_{1} \longrightarrow \overline{\mathcal{C}_{1}}
  $$
  $$
  \phantom{rrrrrrrrrrrr} (t ,\eta) \longmapsto \Lambda(t,\eta):= \frac{2}{\pi}\left( \frac{\pi}{2}- t \right).M + t.\eta$$
$\Lambda(t,\eta)$ is the linear curve starting from the fixed point $M$ at $t = 0$, and arriving to the point $\eta \in \partial\mathcal{C}_{1}$
 at $t = \frac{\pi}{2}$.
For every point $Q\in \overline{\mathcal{C}_{1}}$,
   there exists then a unique pairing
  $$(t,\eta) \in \left[0,\frac{\pi}{2}\right] \times \partial \mathcal{C}_{1}$$ such that the point $Q$ can be written
\begin{equation}\label{ecriture geodesique}
    Q = \Lambda(t,\xi).
\end{equation}
The parameter $t \in \left[0,\frac{\pi}{2} \right]$ represents the Euclidean distance from $Q$ to the fixed point $M$.\\
The observations (1) and (2) enable us to define the second restriction $\varphi_{2}\big/W_{+}$
for every point $P = \Gamma(t,\xi)
\in \overline{U}_{+}$ by:
\begin{equation}\label{deuxieme final}
    \left(\varphi_{2}\big/W_{+}\right)(P) = \left(\varphi_{2}\big/W_{+}\right)\left( \Gamma(\xi,t) \right)=
    \Lambda\left(t, F^{-1}\left(\varphi_{1}(\xi) \right)\right)\in \overline{\mathcal{C}_{1}}.
\end{equation}
In other words, $\varphi_{2}\big/W_{+}$ sends the extremity  $\xi \in \partial W_{+}$ to the extremity
$\eta= F^{-1}\left(\varphi_{1}(\xi)\right)\in \partial\mathcal{C}_{1}$
and sends the point $P = \Gamma(\xi,t)$ of the geodesic arc joining the north pole $N$ to $\xi\in \partial U_{+}$
which is at the geodesic distance
$t$ from $N$ to the point $Q = \Lambda(t,F^{-1}\left(\varphi_{1}(\xi)\right))$ of the segment
 $\left[M,F^{-1}\left(\varphi_{1}(\xi)\right)\right]$ which is at the Euclidean distance $t$ from the fixed point $M \in \mathcal{C}_{1}$.\\
Taking into account the restrictions $\varphi_{2}\big/W_{\sigma}$ and $\varphi_{2}\big/W_{+}$ , we can now define the global homeomorphism
$$
\varphi_{2}: U_{2}=\overline{ W_{+}}\bigcup W_{\sigma} \longrightarrow \overline{\mathcal{C}_{1}}\bigcup \left(\mathcal{C}_{2}-
\overline{\mathcal{C}}_{1} \right)
$$
by the following expression
\begin{equation}\label{psi}
         \varphi_{2}(P): = \left\{\begin{split}
F^{-1}\left( \varphi_{1}(P)\right)& \phantom{rrrrrr} \textrm{if} \phantom{r} P \in W_{\sigma}\\
 \Lambda\left(t,F^{-1}\left(\varphi_{1}(\xi)\right)\right)  & \phantom{rrrrrr} \textrm{if} \phantom{r} P = \Gamma(t,\xi) \in \overline{W_{+}}.
    \end{split}
    \right.
\end{equation}
\begin{remark}$ $\\
\begin{enumerate}
\item Since by construction, the restrictions $\varphi_{2}\big/W_{\sigma}$ and $\varphi_{2}\big/W_{+}$ are homeomorphisms, and
for all $P\in \partial W_{+}\bigcap \partial W_{\sigma}$, we have
$$\left(\varphi_{2}\big/W_{\sigma}\right)(P)=\left(\varphi_{2}\big/W_{+}\right)(P)$$ then $\varphi_{2}$ is a homeomorphism.
\item It follows from the definition of $\varphi_{1}$ and $\varphi_{2}$, that the change of charts in the domain $U_{1}\bigcap U_{2} = W_{\sigma}$,
is the bi-holomorphic mapping
$$ \varphi_{2}\circ \varphi_{1}^{-1} = F^{-1}: \mathcal{B}_{1}-\mathcal{B}_{r_{1}} \longrightarrow \mathcal{C}_{2}-\mathcal{C}_{1}.
$$
 \end{enumerate}
\end{remark}
Then $
\mathcal{A}:= \bigg \{ \left( U_{1}, \varphi_{1} \right), \left(U_{2}, \varphi_{2} \right)    \bigg\}$ is a complex analytic atlas on the sphere $S_{2n}$.
 Hence the sphere $S_{2n}$ admits a complex structure,
 which means that $\left\{ Necs \right\}_{n}$ is false. The implication $ \left\{ Necs \right\}_{n} \Longrightarrow \left\{ Hath\right\}_{n}$ is then true
 and the proof of theorem \ref{equiv} is then complete.
 \end{proof}
 \begin{corollary}\textbf{(Theorem of complex spheres)}
 Among all spheres $S_{2n}$, only the sphere $S_{2}$ admits a complex structure.
 \end{corollary}
 \begin{proof}
 This is a corollary of the main theorem \ref{equiv} and Hartogs' theorem.
 \end{proof}
\section{\textbf{Appendix}}\label{Annexe}
 We summarize below the main tools of pseudoconvexity, and $\check{C}$ech cohomology we need to prove the main theorem \ref{equiv}.
 For pseudoconvexity, see H\"{o}rmander \cite{12}, and for $\check{C}$ech cohomology, see O.forster \cite{7}.
\subsection{Pseudoconvex manifolds}
\subsubsection{Plurisubharmonic functions}$ $  \\
Recall that a function $f: \Omega \subset \mathbb{C}^{n}\longrightarrow \left[-\infty, +\infty\right[$ is said to be plurisubharmonic
in the open set $\Omega$ (we note $f \in Psh(\Omega)$), if $f$ is upper semicontinuous and satisfies the following mean value inequality
\begin{equation}\label{plurisubharmonic}
    f\left( z_{0} \right) \leq \frac{1}{2\pi}\int_{0}^{2\pi}f\left(z_{0}+ e^{i \theta}\xi   \right)d\theta
\end{equation}
for every $\xi\in \mathbb{C}^{n}$ such that $z_{0}+\xi \overline{D} \in \Omega$, where $D$ is the unit disc of $\mathbb{C}$.\\
The following properties will be very useful.\\
1) If $f$ is convex on $\Omega$, then $ f \in Psh(\Omega)$. Indeed, convexity implies continuity, and the inequality of convexity
$$
f\left( z_{0} \right) \leq \frac{1}{2}\left\{\{ f\left(z_{0} + \xi e^{i \theta} \right) +  f\left(z_{0} - \xi e^{i \theta} \right)\right\}
$$
implies inequality of plurisubharmonicity (\ref{plurisubharmonic}). \\
2) Let $F: \Omega_{1}\subset \mathbb{C}^{n}\longrightarrow \Omega_{2}\subset \mathbb{C}^{n}$ be a bi-holomorphic mapping, then
$$ g \in Psh(\Omega_{2}) \Longleftrightarrow g \circ F \in Psh(\Omega_{1}). $$
3) Let $X$ be a complex manifold,
 a function $f : X \longrightarrow \left[-\infty , +\infty\right[$ is said to be plurisubharmonic on $X$ if for
 every complex chart $(U,\varphi)$, the function $$ f \circ \varphi^{-1}:\varphi (U)\longrightarrow \left[-\infty  , + \infty\right[$$ is plurisubharmonic.
\begin{remark}
By property 2) above, the notion of plurisubharmonicity on a complex manifold is independent of the choice of the chart $(U,\varphi)$.
\end{remark}
\subsubsection{Pseudoconvex manifols} $ $ Let $X$ be a complex manifold, and let $f : X \longrightarrow \mathbb{R}$ be a smooth function. $f$ is said to
 be an exhaustion on $X$ is for every $c \in \mathbb{R}$, the sublevel set $X_{c} = f^{-1}(c)$ is relatively compact in $X$.\\
 The complex manifold $X$ is said to be strongly pseudoconvex if there exists a smooth strictly psh function $f$ exhaustion on $X$. We write in this case
 $$
 X = \bigcup_{c \in \mathbb{R}}f^{-1}(c).
 $$
 \begin{theorem}\label{Cartan theorem}\textbf{(Theorem B of Cartan.)}
 If $X$ is a strongly pseudoconvex manifold, then for every $ q\geq 1$, $\mathcal{H}^{0,q}(X, \mathbb{C}^{n}) = \left\{ 0 \right\}$.
 \end{theorem}
\subsection{$\check{\textbf{C}}$ech cohomology}
We recall below the necessary tools from $\check{C}$ech cohomology, we have also to use in the proof of the main theorem \ref{equiv}. \\
Let $X$ be a topological space, and let $\mathfrak{V}= \left\{V_{j} \right\}_{j \in J}$ be a covering of $X$, and let $\mathcal{G}$ be a sheaf of abelian groups on $X$. We define the $q^{th}$ group of cochains of $\mathcal{G}$ with respect to the covering $\mathfrak{V}$, by:
$$
C^{q}(\mathfrak{V}, \mathcal{G}): = \prod_{(j_{0},...,j_{q})\in J^{q+1}}\mathcal{G}(V_{j_{0}}\cap ...\cap V_{j_{q}}).
$$
A $q-$cochain is then a family
$$ \left(f_{j_{0},...,j_{q}}\right)_{(j_{0},...,j_{q})\in J^{q+1}} \phantom{rrr} \textrm{such that}
\phantom{r}  f_{j_{0},...,j_{q}} \in \mathcal{G}\left( V_{j_{0}} \cap ... \cap V_{j_{q}}\right) $$ for all $(j_{0},...,j_{q})\in J^{q+1}$.\\
We define the coboundary operators
$$ \delta : C^{0}(\mathfrak{V}, \mathcal{G})  \longrightarrow  C^{1}(\mathfrak{V}, \mathcal{G})$$
and
$$ \delta : C^{1}(\mathfrak{V}, \mathcal{G})  \longrightarrow  C^{2}(\mathfrak{V}, \mathcal{G})$$
as follows:\\
1) For $(f_{i})_{i\in J} \in  C^{0}(\mathfrak{V}, \mathcal{G})$, we define $ \delta \big((f_{i})_{i\in J}\big) = (g_{i,j})_{(i,j)\in J^{2}} \in  C^{1}(\mathfrak{V}, \mathcal{G}) $ by
$$ g_{i,j} = f_{j} - f_{i} \phantom{r} \in \mathcal{G}(V _{i} \cap V_{j}).$$
2) For $(f_{i,j}\big)_{(i,j)\in J^{2}}\in C^{1}\bigg( \mathfrak{V}, \mathcal{G}\bigg)$, we define $\delta((f_{i,j}\big)_{(i,j)\in J^{2}})
 = (g_{i,j,k}\big)_{(i,j,k)\in J^{3}} $ by
$$
g_{i,j,k} = f_{j,k}-f_{i,k}+f_{i,j}\phantom{r} \in \mathcal{G}(V_{i}\cap V_{j}\cap V_{k}).
$$
Since the boundary operators are group homomorphisms, we are lead then to consider the group of 1-cycles
$$
\mathcal{Z}^{1}\bigg( \mathfrak{V} , \mathcal{G} \bigg): = ker\phantom{r} C^{1}\bigg(  \mathfrak{V}, \mathcal{G} \bigg) \stackrel{\delta}{\longrightarrow}  C^{2}\bigg(  \mathfrak{V}, \mathcal{G} \bigg)
$$
and the group of 1-coboundaries
$$
\mathcal{B}^{1}\bigg(  \mathfrak{V}, \mathcal{G} \bigg): = Im \phantom{r} C^{0}\bigg( \mathfrak{V}, \mathcal{G} \bigg) \stackrel{\delta}{\longrightarrow}  C^{1}\bigg( \mathfrak{V}, \mathcal{G} \bigg)
$$
and to define with respect to the covering $\mathfrak{V}$, the cohomology group with coefficients in the sheaf $\mathcal{G}$, that is the following quotient group
$$
\mathcal{H}^{1}\bigg( \mathfrak{V}, \mathcal{G}\bigg): =
\mathcal{Z}^{1}\bigg(  \mathfrak{V}, \mathcal{G} \bigg)\bigg/
\mathcal{B}^{1}\bigg(  \mathfrak{V}, \mathcal{G} \bigg).
$$
We need about the cohomology group $\mathcal{H}^{1}\bigg( \mathfrak{V}, \mathcal{G}\bigg)$  with coefficients in the sheaf $\mathcal{G}$, to recall the following two fundamental facts.\\
1) The cohomology group $\mathcal{H}^{1}\bigg( \mathfrak{V}, \mathcal{G}\bigg)$ depends obviously of the covering $\mathfrak{V}$, but we can define the cohomology group $\mathcal{H}^{1}\bigg( X, \mathcal{G} \bigg)$
of the topological space $X$ as the inductive limit of the groups $\mathcal{H}^{1}\bigg( \mathfrak{V}, \mathcal{G} \bigg)$ over all finer coverings $\mathfrak{V}$ of $X$, that is
$$
\mathcal{H}^{1}\bigg(X,\mathcal{G}\bigg): = \lim_{\mathfrak{V}}\mathcal{H}^{1}\bigg( \mathfrak{V}, \mathcal{G} \bigg).
$$
2) To compute the cohomology group $\mathcal{H}^{1}\bigg( X, \mathcal{G}\bigg)$ from the group $\mathcal{H}^{1}\bigg( \mathfrak{V}, \mathcal{G}\bigg)$,
we use the following powerful theorem.
\begin{theorem}\label{Leray}\textbf{(Leray)}
Let $X$ be a topological space, and let $\mathcal{G}$ be a sheaf over $X$, and let $\mathfrak{V} = \left\{V_{j} \right\}_{j \in J}$ be a covering of $X$ such that for all $j \in J$, $\mathcal{H}^{q}(V_{j},\mathcal{G}) = \left\{ 0\right\}$. Then
$$
\mathcal{H}^{q}(X,\mathcal{G}) = \mathcal{H}^{q}(\mathfrak{V},\mathcal{G}).
$$
\end{theorem}


\begin{thebibliography}{99}
\bibitem{1} Borel A. and Serre J.-P. \textit{ Groupes de Lie et puissances r´eduites de Steenrod} Amer. J. Math.
75 (1953), 409–448. MR0058213 2, 9
\bibitem{2}  Bryant R.\textit{On the geometry of almost complex 6-manifolds, Asian J. Math.} 10 (2006), 561–605.
MR2253159 5, 6, 11
\bibitem{3}Bryant R. \textit{S-S Chern's study of almost complex structure on the six-sphere}
arXiv:1405-3405, 2014.
\bibitem{4} Ehrenpreis L. \textit{A new proof and an extension of Hartog's theorem} Bull.Amer.Math.Soc.67(1961),507-509.
\bibitem{5}Ehresmann C. \textit{ Sur la th\'{e}ori des espaces fibr\'{e}s }
Colloq. Int. C.N.R.S. Top. Alg, Paris, 1947, 3-35
\bibitem{6}Ehresmann C. \textit{ Sur les vari\'{e}t\'{e}s presque complexes }
Proc. Int. Math. Cong. 1950. P.412-419.
\bibitem{7} Forster O. \textit{Lectures on Rieman surfaces}
Springer Verlag. 1981.
\bibitem{8}Gray A. \textit{A property of a hypothetical complex structure on the six sphere}
Boll. Un. Math. Ital. B(7) 11(2, suppl.): 251-255, 1997.
\bibitem{9}Gunning R C, Rossi H. \textit{Analytic functions of several complex variables}
Prentice Hall 1965.
\bibitem{10} Hartogs F. \textit{Einige Folgerungen aus der Cauchyschen Integralformel bei Funktionen mehrer Ver\"{a}nderlicher.}
M\"{u}nchener Sitz. Ber. 36(1906), 223-242.
\bibitem{11}Hopf H. \textit{ Sur les champs d'\'{e}l\'{e}ments de surfaces dans les vari\'{e}t\'{e}s \`{a} quatre dimensions}
Colloq. Int. C.N.R.S. Top. Alg, Paris, 1947, 55-59
\bibitem{12}H\"{o}rmander, L. \textit{An Introduction to complex analysis in several variables}
3rd revised edition, North Holland Math Library, Vol 7 Amsterdam.
\bibitem{13}Kirchhoff A. \textit{}
C. R. Acad. Sc. t. 223, 1948, p. 1258-1260.
\bibitem{14}Krantz S G. \textit{Fonction theory of several complex variables}
American Mathemathical Society 2000.
\bibitem{15} LeBrun C. \textit{ Orthogonal complex structures on S6. }
Proceedings of the AMS 101 (1987), 136–138. MR0897084 2, 9.
\bibitem{16}Libermann P. \textit{Sur les strutures presque complexes et autres structures infinit\'{e}simales r\'{e}guli\`{e}res}
Bull. Soc. Math. France. Tome 83, p.185-224.
\bibitem{17} Newlander A. and L. Nirenberg L. \textit{ Complex analytic coordinates in almost complex manifolds}
Ann. of Math. 65 (1957), 391–404. 2, 8
\bibitem{18}Range R M. \textit{Holomorphic functions and integral representations in several variables}
Springer Verlag.
\bibitem{19}Ughart L. \textit{Hodge numbers of a hypothetical complex structure on the six sphere}
Geom. Dedicata. 81(1-3):173-179,2000.
\end{thebibliography}
\end{document}